\documentclass[onefignum,onetabnum]{siamart190516}

\usepackage{amssymb}%
\usepackage{dsfont}%

\usepackage{lipsum}
\usepackage{amsfonts}
\usepackage{graphicx}
\usepackage{epstopdf}
\usepackage{algorithmic}
\ifpdf
  \DeclareGraphicsExtensions{.eps,.pdf,.png,.jpg}
\else
  \DeclareGraphicsExtensions{.eps}
\fi

\newsiamremark{remark}{Remark}
\newsiamremark{hypothesis}{Hypothesis}
\crefname{hypothesis}{Hypothesis}{Hypotheses}
\newsiamthm{claim}{Claim}

\headers{Small-gain theory for infinite networks}{C.~Kawan, M.~Zamani}

\title{A small-gain theory for infinite networks via infinite-dimensional gain operators\thanks{This work was supported in part by the German Research Foundation (DFG) through the grant ZA 873/4-1 and the H2020 ERC Starting Grant
AutoCPS (grant agreement No. 804639).}}

\author{Christoph Kawan\thanks{Institute of Informatics, LMU Munich, 80538 M\"{u}nchen, Germany (\email{christoph.kawan@lmu.de}).}
\and Majid Zamani\thanks{Computer Science Department, University of Colorado Boulder, CO 80309, USA. M. Zamani is also
with the Institute of Informatics, LMU Munich, Germany (\email{majid.zamani@colorado.edu}).}}

\usepackage{amsopn}

\newtheorem{assumption}{Assumption}%

\ifpdf
\hypersetup{
  pdftitle={A small-gain theory for infinite networks via infinite-dimensional gain operators},
  pdfauthor={C.~Kawan and M.~Zamani}
}
\fi

\allowdisplaybreaks

\newcommand{\R}{\mathbb{R}}%
\newcommand{\UC}{\mathcal{U}}%
\newcommand{\N}{\mathbb{N}}%
\newcommand{\Z}{\mathbb{Z}}%
\newcommand{\IC}{\mathcal{I}}%
\newcommand{\KC}{\mathcal{K}}%
\newcommand{\LC}{\mathcal{L}}%
\newcommand{\PC}{\mathcal{P}}%
\newcommand{\ep}{\varepsilon}%
\newcommand{\unit}{\mathds{1}}%
\newcommand{\dist}{\mathrm{dist}}%
\newcommand{\id}{\mathrm{id}}%
\newcommand{\inner}{\mathrm{int}}%
\newcommand{\NC}{\mathcal{N}}%
\newcommand{\rmD}{\mathrm{D}}%
\newcommand{\tm}{\times}%
\newcommand{\esssup}{\operatorname*{ess\;sup}}%
\newcommand{\Fix}{\mathrm{Fix}}%

\begin{document}

\maketitle

\begin{abstract}
In this paper, we develop a new approach to study gain operators built from the interconnection gains of infinite networks of dynamical systems. Our focus is on the construction of paths of strict decay which are used for building Lyapunov functions for the network and thus proving various stability properties, including input-to-state stability. Our approach is based on the study of an augmented gain operator whose fixed points are precisely the points of decay for the original gain operator. We show that plenty of such fixed points exist under a uniform version of the no-joint-increase condition. Using these fixed points to construct a path of strict decay, in general, requires specific dynamical properties of associated monotone operators. For particular types of gain operators such as max-type operators and subadditive operators, these properties follow from uniform global asymptotic stability of the system induced by the original gain operator. This is consistent with former results in the literature which can readily be recovered from our theory.
\end{abstract}



\section{Introduction}

Recent advances in computation and communication capabilities and low costs of sensors and actuators have led to the emergence of large-scale interconnected networks including traffic, transportation, social, and power networks, to name a few. Although these advances result in more sophisticated networks exhibiting emergent behaviors, we face new analysis and design challenges that we do not know how to solve with existing techniques. For example, stability properties of large-scale networks may deteriorate as the number of participating agents increases \cite{bamieh}.%

Infinite networks have been introduced as over-approximations of large-but-finite networks as worst-case scenarios. In the past few years, several results have been obtained for input-to-state stability (ISS) analysis of infinite networks with nonlinear components \cite{kawan,kawan2021lyapunov,mironchenko2021iss,mironchenko2021nonlinear,dashkovskiy2020stability}. In those results, the influence of any subsystem on other subsystems in the network is characterized by so-called gain functions. The gain operator constructed from these functions describes the interconnection structure of the network. The small-gain theorems proposed in \cite{kawan,kawan2021lyapunov,mironchenko2021iss,mironchenko2021nonlinear} and the construction of associated ISS Lyapunov functions are based on the assumption that a so-called path of strict decay exists
for the gain operator. Unfortunately, constructing these paths of strict decay is a challenging problem.%

In the small-gain theory for input-to-state stability of finite networks, a central tool used for the construction of a path of strict decay is a topological result known as the KKM lemma. More precisely, the KKM lemma is used to guarantee the existence of sufficiently many points of strict decay. A piecewise linear path of strict decay is then constructed via connecting certain points of strict decay by straight lines, see \cite{dashkovskiy2010small}. This construction breaks down in the infinite-dimensional ordered Banach space $(\ell_{\infty},\ell_{\infty}^+)$ on which the gain operator of an infinite network acts within the framework proposed in \cite{dashkovskiy2020stability,kawan2021lyapunov,mironchenko2021iss}. First, existing infinite-dimensional versions of the KKM lemma are not well suited to the geometry of $\ell_{\infty}$. Second, the piecewise linear construction is heavily based on local compactness. Hence, we need to develop a completely different approach to handle infinite networks.%

In \cite{kawan2021lyapunov}, the authors treated infinite networks in which the influences of the subsystems on each other are described by a max-type gain operator. This is the case if the influence on a fixed subsystem by its neighbors can be expressed as the maximum of the associated gain functions. In this setup, a small-gain theorem was proved for the construction of an ISS Lyapunov function for the network. The existence of a path of strict decay is here implied by the uniform global asymptotic stability of the discrete-time system induced by a slightly enlarged gain operator together with some uniform local Lipschitz estimates for the gain functions; in this case, the path can be constructed via the strong transitive closure of the (enlarged) gain operator $\Gamma:\ell_{\infty}^+ \rightarrow \ell_{\infty}^+$, which is given by%
\begin{equation*}
  Q(s) = \bigoplus_{k=0}^{\infty}\Gamma^k(s),
\end{equation*}
where $\oplus$ stands for the componentwise supremum of vectors in $\ell_{\infty}^+$. This approach originates from \cite{dashkovskiy2019stability}.%

In this paper, we study gain operators defined from nonlinear gains, which are not necessarily of maximum type. Instead, we consider gain operators defined via so-called monotone aggregation functions introduced in \cite{dashkovskiy2010small} for finite networks. A path of strict decay for the gain operator $\Gamma$ is a mapping $\sigma:\R_+ \rightarrow \ell_{\infty}^+$ whose essential property is that $\Gamma \circ \sigma \leq (\id + \rho)^{-1} \circ \sigma$ for a $\KC_{\infty}$-function $\rho$. For the enlarged gain operator $\Gamma_{\rho} := (\id + \rho) \circ \Gamma$, this simplifies to $\Gamma_{\rho} \circ \sigma \leq \sigma$. A necessary condition for the existence of $\sigma$ is the existence of sufficiently many points of decay for $\Gamma_{\rho}$. The main observation in this paper is that a point $s \in \ell_{\infty}^+$ is a point of decay for a gain operator $\Gamma$ if and only if it is a fixed point of the augmented gain operator%
$$\hat{\Gamma}(s) := s \oplus \Gamma(s),\quad \hat{\Gamma}:\ell_{\infty}^+ \rightarrow \ell_{\infty}^+.$$
Hence, we study the operator $\hat{\Gamma}$ and provide conditions under which the mapping%
\begin{equation*}
  \sigma(r) := \bigoplus_{k=0}^{\infty}\hat{\Gamma}^k(r\unit),\quad \sigma:\R_+ \rightarrow \ell_{\infty}^+%
\end{equation*}
is a path of strict decay for $\Gamma$, where $\unit \in \ell_{\infty}^+$ is the vector whose components are all equal to $1$.%

A sufficient condition for $\sigma$ to be well-defined and satisfy some necessary uniformity assumptions is that the discrete-time system induced by $\hat{\Gamma}$ is uniformly globally stable. This, in turn, is implied by various uniform small-gain conditions studied before in the literature for both finite and infinite networks, see, e.g.~\cite{mironchenko2021nonlinear,ruffer2010monotone}. From the construction of $\sigma$, it turns out that $\sigma(r)$ is the minimal fixed point of the monotone operator $\Gamma_r(s) = r\unit \oplus \Gamma(s)$. Our main result shows that a locally Lipschitz continuous single-valued selection of the set-valued map $r \mapsto \Fix(\Gamma_r)$, where $\Fix(\Gamma_r)$ is the fixed point set of $\Gamma_r$, is a path of strict decay provided that the system induced by a slightly enlarged $\Gamma$ is uniformly globally asymptotically stable and some mild regularity assumptions hold. We specialize this result to different classes of gain operators, including subadditive homogeneous operators, finite-dimensional operators and sum-type operators, partially recovering known results in the literature.%

In the most general case, it remains still an open problem under which conditions we can construct a path of strict decay with our approach, since the existence of a locally Lipschitz continuous single-valued selection of $r \mapsto \Fix(\Gamma_r)$ is hard to check. In all of our results specialized to particular types of gain operators, the assumptions imply that $\Gamma_r$ only admits one fixed point for each $r \geq 0$, which is the simplest case. In those cases, nevertheless, we think that our approach is more intuitive and constructive than the approach for finite networks based on the KKM lemma, which is a pure existence result.%

For further reading on input-to-state stability of finite and infinite networks, we refer the reader to \cite{KaJ11,MiP20,Sch20} and the references given therein.%

The paper is organized as follows. In Section \ref{sec_prelim}, we introduce the central concepts of the paper and explain in detail the meaning of a path of strict decay for the small-gain approach to input-to-state stability. Section \ref{sec_augmented_op} introduces the augmented gain operator and studies its properties with a special emphasis on the stability properties of the induced discrete-time system. In Section \ref{sec_ugas}, we try to understand uniform global asymptotic stability of the system induced by a gain operator in terms of the behavior under one iterate of the operator. This leads to a certain small-gain-type condition which turns out to be very useful in the construction of a path of strict decay. Finally, Section \ref{sec_path} is devoted to the construction of paths of strict decay and contains the main result of the paper (cf. Theorem \ref{thm_path_of_decay}.)%

\section{Preliminaries}\label{sec_prelim}

In this section, we introduce notation and some fundamental definitions.%

\paragraph{Notation} By $\N = \{1,2,3,\ldots\}$, we denote the set of natural numbers and $\Z_+ := \N \cup \{0\}$. The Banach space of all bounded real sequences $\ell_{\infty}$ is equipped with the norm $\|s\|_{\infty} = \sup_{i\in\N}|s_i|$ for any $s = (s_i)_{i\in\N} \in \ell_{\infty}$. We write $\pi_i:\ell_{\infty} \rightarrow \R$ for the projection to the $i$th component, $\pi_i(s) = s_i$. By%
\begin{equation*}
  \ell_{\infty}^+ := \{ s = (s_i)_{i\in\N} \in \ell_{\infty} : s_i \geq 0,\ \forall i \in \N \},%
\end{equation*}
we denote the standard positive cone in $\ell_{\infty}$. This cone is closed and has a nonempty interior. It induces an order on $\ell_{\infty}$ by $s^1 \leq s^2$ if and only if $s^2 - s^1 \in \ell_{\infty}^+$. Given $s^1,s^2 \in \ell_{\infty}^+$, we write $s = s^1 \oplus s^2$ for the vector with components $s_i = \max\{s^1_i,s^2_i\}$ for all $i \in \N$. We also extend this definition to an arbitrary number of vectors $s^{\alpha}$ ($\alpha\in A$), and write $\bigoplus_{\alpha \in A} s^{\alpha} = (\sup_{\alpha\in A}s_i^{\alpha})_{i\in\N}$. By $e_i$, we denote the $i$th unit vector in $\ell_{\infty}$ and by $\unit$ the vector all of whose components are equal to $1$, i.e.~$\unit = \sum_{i\in\N}e_i$. If $(X,\|\cdot\|_X)$ is a normed space, we write $B_{\ep}(x) = \{ y \in X : \|x - y\|_X < \ep \}$ and $\dist(x,A) = \inf_{y\in A}\|x - y\|_X$ for any $A \subseteq X$ and $x \in X$.%

We further use the following classes of comparison functions:%
\begin{align*}
  \PC &:= \left\{\gamma \in C^0(\R_+,\R_+): \gamma(0) = 0,\ \gamma(r) > 0,\ \forall r > 0 \right\},\\
  \KC &:= \left\{\gamma \in \PC: \ \gamma\mbox{ is strictly increasing}\right\}, \\
  \KC_{\infty} &:= \left\{\gamma\in\KC:\ \gamma\mbox{ is unbounded}\right\}, \\
  \LC &:= \big\{\gamma \in C^0(\R_+,\R_+):\ \gamma\mbox{ is strictly decreasing with } \lim_{t\rightarrow\infty}\gamma(t)=0 \big\}, \\
  \KC\LC &:= \{\beta \in C^0(\R_+^2,\R_+): \beta(\cdot,t)\in\KC,\ \forall t \geq 0,\ \beta(r,\cdot)\in \LC,\ \forall r > 0\}.%
\end{align*}

A function $\mu:\ell_{\infty}^+ \rightarrow [0,\infty]$ is called a \emph{monotone aggregation function (MAF)} if it satisfies the following properties:%
\begin{enumerate}
\item[(i)] $\mu(0) = 0$.%
\item[(ii)] If $s^1 \leq s^2$, then $\mu(s^1) \leq \mu(s^2)$ for any $s^1,s^2 \in \ell_{\infty}^+$.%
\item[(iii)] For each $I \subseteq \N$, define $\ell_{\infty}^+(I) := \{ s \in \ell_{\infty}^+ : s_i = 0, \forall i \in \N \setminus I \}$. Then, $\mu$ is finite-valued and continuous on $\ell_{\infty}^+(I)$ for every finite $I$.%
\end{enumerate}
Sometimes, MAFs are also required to be subadditive (see, e.g., \cite{ruffer2010monotone}). However, for most results in this paper we do not need such an assumption.%

An operator $T:\ell_{\infty}^+ \rightarrow \ell_{\infty}^+$ is called monotone if $s^1 \leq s^2$ implies $T(s^1) \leq T(s^2)$ for all $s^1,s^2 \in \ell_{\infty}^+$. For any monotone operator $T$, we call the discrete-time system%
\begin{equation}\label{eq_general_monotone_sys}
  s(k+1) = T(s(k)),\quad k = 0,1,2,\ldots%
\end{equation}
\begin{itemize}
\item \emph{uniformly globally stable (UGS)} if there exists $\varphi \in \KC_{\infty}$ with%
\begin{equation*}
  \|T^k(s)\|_{\infty} \leq \varphi(\|s\|_{\infty}) \mbox{\quad for all\ } s \in \ell_{\infty}^+,\ k \geq 0.%
\end{equation*}
\item \emph{uniformly globally asymptotically stable (UGAS)} if there exists $\beta \in \KC\LC$ with%
\begin{equation*}
  \|T^k(s)\|_{\infty} \leq \beta(\|s\|_{\infty},k) \mbox{\quad for all\ } s \in \ell_{\infty}^+,\ k \geq 0.%
\end{equation*}
\item \emph{uniformly globally exponentially stable (UGES)} if there exist $M>0$ and $\gamma\in(0,1)$ with%
\begin{equation*}
  \|T^k(s)\|_{\infty} \leq M\gamma^k\|s\|_{\infty} \mbox{\quad for all\ } s \in \ell_{\infty}^+,\ k \geq 0.%
\end{equation*}
\item \emph{globally attractive} if for every $s \in \ell_{\infty}^+$%
\begin{equation*}
  \lim_{k \rightarrow \infty} \|T^k(s)\|_{\infty} = 0.%
\end{equation*}
\item \emph{globally componentwise attractive} if for every $s \in \ell_{\infty}^+$%
\begin{equation*}
  \lim_{k \rightarrow \infty} \pi_i \circ T^k(s) = 0 \mbox{\quad for all\ } i \in \N.%
\end{equation*}
\end{itemize}

\paragraph{Infinite networks} The basis of the study carried out in this paper is the following setup: We have a network of countably many finite-dimensional continuous-time control systems%
\begin{equation*}
  \Sigma_i: \quad \dot{x}_i = f_i(x_i,\bar{x}_i,u_i),\quad i \in \N,%
\end{equation*}
where $x_i \in \R^{n_i}$, $u_i \in \R^{m_i}$, $\bar{x}_i = (x_j)_{j \in I_i}$ and $I_i \subset \N \setminus \{i\}$ is a \emph{finite} index set (the set of \emph{neighbors} of system $\Sigma_i$). We further equip each $\R^{n_i}$ and each $\R^{m_i}$ with a norm. Although these norms can depend on the index $i$, we simply write $|\cdot|$ for each of them.%

Under conditions specified in \cite[Thm.~II.1]{kawan2021lyapunov}, we can aggregate the systems $\Sigma_i$ to obtain an infinite-dimensional system%
\begin{equation*}
  \Sigma: \quad \dot{x} = f(x,u)%
\end{equation*}
which is well-posed on the state space%
\begin{equation*}
  X = \ell_{\infty}(\N,(n_i)) := \bigl\{ x = (x_i)_{i\in\N} : x_i \in \R^{n_i},\ \sup_{i\in\N}|x_i| < \infty\bigr\}%
\end{equation*}
with the input space $U = \ell_{\infty}(\N,(m_i))$ and class of external input functions given by%
\begin{equation*}
  \UC := \{u \in L_{\infty}(\R_+,U) : u \mbox{ is piecewise right-continuous} \},
\end{equation*}
which is equipped with the $L_{\infty}$-norm $\|u\|_{\UC} := \esssup_{t\in\R_+} |u(t)|_U$. That is, for every initial state $x^0 \in X$ and every $u \in \UC$, there exists a unique solution of the ODE $\dot{x}(t) = f(x(t),u(t))$ with initial condition $x(0) = x^0$ in the sense of Carath\'eodory (see \cite{kawan2021lyapunov}), and every bounded maximal solution is defined on $\R_+$.%

We further assume that there exist $\psi_1,\psi_2 \in \KC_{\infty}$ and $\alpha \in \PC$ as well as continuous functions $V_i:\R^{n_i} \rightarrow \R_+$, $i\in\N$, which are $C^1$ outside of $0$, such that for every $i \in \N$ the following properties are satisfied:%
\begin{itemize}
\item For all $x_i \in \R^{n_i}$, we have $\psi_1(|x_i|) \leq V_i(x_i) \leq \psi_2(|x_i|)$.%
\item There exist $\gamma_{ij} \in \KC \cup \{0\}$ for all $j\in\N$, where $\gamma_{ij} = 0$ whenever $j \in \N \setminus I_i$, $\gamma_{iu} \in \KC$, and a MAF $\mu_i:\ell_{\infty}^+ \rightarrow [0,\infty]$ such that for all $x = (x_j)_{j\in\N} \in X$ and $u = (u_j)_{j\in\N} \in U$ the following implication holds:%
\begin{align*}
  V_i(x_i) &> \max\{\mu_i( [\gamma_{ij}(V_j(x_j))]_{j\in\N} ), \gamma_{iu}(|u_i|) \} \\
	         &\qquad \Rightarrow \nabla V_i(x_i) f_i(x_i,\bar{x}_i,u_i) \leq -\alpha(V_i(x_i)).%
\end{align*}
\end{itemize}
The function $V_i$ is called an \emph{ISS Lyapunov function} for the subsystem $\Sigma_i$, the $\gamma_{ij}$ are called \emph{internal gains}, and the $\gamma_{iu}$ \emph{external gains}.%

The internal gains $\gamma_{ij}$ together with the MAFs $\mu_i$ give rise to a \emph{gain operator}, defined by%
\begin{equation*}
  \Gamma(s) := (\mu_i( [\gamma_{ij}(s_j)]_{j\in\N} ))_{i\in\N},\quad \Gamma:\ell_{\infty}^+ \rightarrow \ell_{\infty}^+,%
\end{equation*}
which we always assume to be \emph{well-defined and continuous}.%

\begin{assumption}\label{ass_standing}
The gain operator $\Gamma:\ell_{\infty}^+ \rightarrow \ell_{\infty}^+$ is well-defined and continuous.
\end{assumption}

Observe that $\Gamma(0) = 0$ and that $\Gamma$ is a monotone operator due to property (ii) of MAFs and the fact that each $\gamma_{ij}$ is a monotonically non-decreasing function:%
\begin{equation*}
  s^1 \leq s^2 \quad \Rightarrow \quad \Gamma(s^1) \leq \Gamma(s^2) \mbox{\quad for all\ } s^1,s^2 \in \ell_{\infty}^+.%
\end{equation*}

It is not easy to provide a general characterization of the well-definedness and continuity of gain operators. However, we can provide the following sufficient condition for $\Gamma$ to be well-defined.%

\begin{proposition}\label{prop_gamma_welldefined}
Assume that the family $\{\gamma_{ij} : i,j\in\N\}$ is pointwise equicontinuous and for every $R > 0$, it holds that%
\begin{equation}\label{eq_wd_ass}
  \sup_{i \in \N} \mu_i\Bigl( R\sum_{j\in I_i}e_j \Bigr) < \infty.%
\end{equation}
Then $\Gamma$ is well-defined.
\end{proposition}

\begin{proof}
As the proof of \cite[Prop.~4.7]{mironchenko2021nonlinear} shows, the assumption of pointwise equicontinuity implies that for every $s \in \ell_{\infty}^+$ there exists $R > 0$ with $\gamma_{ij}(s_j) \leq R$ for all $i,j \in \N$. Hence, the vector $(\gamma_{ij}(s_j))_{j\in\N}$ can be bounded by the vector $R\sum_{j\in I_i}e_j$. Together with Assumption \eqref{eq_wd_ass} and the monotonicity of the MAFs, this implies $\|\Gamma(s)\|_{\infty} < \infty$. Hence, $\Gamma$ is well-defined.
\end{proof}

We will pay special attention to the following classes of gain operators:%
\begin{itemize}
\item If $\mu_i(s) = \sup_{j\in\N}s_j$ for all $i\in\N$, we call $\Gamma$ a \emph{max-type gain operator}. By \cite[Prop.~4.7]{mironchenko2021nonlinear}, a max-type gain operator is well-defined and continuous if the family $\{ \gamma_{ij} : i,j\in \N \}$ is pointwise equicontinuous. An important property of max-type gain operators is that%
\begin{equation*}
  \Gamma(s^1 \oplus s^2) = \Gamma(s^1) \oplus \Gamma(s^2) \mbox{\quad for all\ } s^1,s^2 \in \ell_{\infty}^+.%
\end{equation*}
Operators with this property are also called \emph{max-preserving}.%
\item If $\mu_i(s) = \sum_{j\in\N}s_j$ for all $i\in\N$, we call $\Gamma$ a \emph{sum-type gain operator}. Sufficient conditions for $\Gamma$ being well-defined and continuous can be found in \cite[Ass.~4.10 and Prop.~4.12]{mironchenko2021nonlinear}.%
\item Assume that all gains are linear functions and each $\mu_i$ is subadditive and homogenenous, i.e.%
\begin{itemize}
\item $\mu_i(rs) = r\mu_i(s)$ for all $r \geq 0$ and $s \in \ell_{\infty}^+$.%
\item $\mu_i(s^1 + s^2) \leq \mu_i(s^1) + \mu_i(s^2)$ for all $s^1,s^2 \in \ell_{\infty}^+$.%
\end{itemize}
Then also $\Gamma$ is subadditive and homogeneous, where the subadditivity of $\Gamma$ has to be understood with respect to the order induced by $\ell_{\infty}^+$, i.e.~$\Gamma(s^1 + s^2) \leq \Gamma(s^1) + \Gamma(s^2)$ for all $s^1,s^2 \in \ell_{\infty}^+$. We then call $\Gamma$ a \emph{subadditive and homogeneous gain operator}. By \cite{mironchenko2021iss}, $\Gamma$ is well-defined and continuous if and only if $\sup_{i\in\N}\mu_i((\gamma_{ij})_{j\in\N}) < \infty$.%
\item If $\Gamma$ is a sum-type operator and all $\gamma_{ij}$ are linear, then $\Gamma$ is the restriction of a linear operator on $\ell_{\infty}$. In this case, we call $\Gamma$ a \emph{linear gain operator}.%
\end{itemize}

We are interested in decay properties of gain operators. The following definitions are central to our investigations.%

\begin{definition}
A point $s \in \ell_{\infty}^+ \setminus \{0\}$ is called a \emph{point of decay} for $\Gamma$ if $\Gamma(s) \leq s$.
\end{definition}

\begin{definition}\label{def_path_of_decay}
A mapping $\sigma:\R_+ \rightarrow \ell_{\infty}^+$ is called a \emph{path of strict decay (for $\Gamma$)}, if all the following properties hold:%
\begin{enumerate}
\item[(P1)] There exists a function $\rho \in \KC_{\infty}$ such that%
\begin{equation*}
  \Gamma(\sigma(r)) \leq (\id + \rho)^{-1} \circ \sigma(r) \mbox{\quad for all\ } r \geq 0,%
\end{equation*}
where $(\id + \rho)^{-1}$ is applied componentwise.%
\item[(P2)] There exist $\sigma_{\min},\sigma_{\max} \in \KC_{\infty}$ satisfying%
\begin{equation*}
  \sigma_{\min} \leq \sigma_i \leq \sigma_{\max} \mbox{\quad for all\ } i \in \N.%
\end{equation*}
\item[(P3)] Each component function $\sigma_i = \pi_i \circ \sigma$, $i\in\N$, is a $\KC_{\infty}$-function.%
\item[(P4)] For every compact interval $J \subset (0,\infty)$, there exist $0 < c \leq C < \infty$ such that for all $r_1,r_2 \in J$ and $i \in \N$%
\begin{equation}\label{eq_local_bilip}
  c|r_1 - r_2| \leq |\sigma_i^{-1}(r_1) - \sigma_i^{-1}(r_2)| \leq C|r_1 - r_2|.%
\end{equation}
\end{enumerate}
If all of the above properties are satisfied with the exception that $\rho = 0$ in (i), we call $\sigma$ a \emph{path of decay (for $\Gamma$)}.
\end{definition}

\begin{remark}\label{rem_path_of_decay}
It is important to note that any path of decay for the scaled gain operator $\Gamma_{\rho} := (\id + \rho) \circ \Gamma$ is a path of strict decay for $\Gamma$. Since the scaled gain operator $\Gamma_{\rho}$ can be regarded as an unscaled gain operator with MAFs $(\id + \rho) \circ \mu_i$ instead of $\mu_i$, it thus suffices to find conditions for the existence of a path of decay.
\end{remark}

\begin{remark}\label{rem_pod2}
According to the proof of \cite[Thm.~VI.1]{kawan2021lyapunov}, the function $\sigma_i^{-1}$ in condition \eqref{eq_local_bilip} can be replaced by $\sigma_i$. That is, the uniform local Lipschitz condition for the functions $\sigma_i^{-1}$ is equivalent to the same condition for the functions $\sigma_i$.
\end{remark}

\paragraph{Small-gain theorem} The importance of the concept of a path of strict decay becomes clear through the small-gain theorem for input-to-state stability presented below. First, we recall the definitions of ISS and ISS Lyapunov functions.%

\begin{definition}\label{def_ISS}
The network $\Sigma$ is called \emph{input-to-state stable (ISS)} if it is forward complete (i.e.~every maximal solution is defined on $\R_+$) and there exist $\beta \in \KC\LC$ and $\gamma \in \KC_{\infty}$ such that%
\begin {equation*}
  \|\phi(t,x,u)\|_X \leq \beta(\|x\|_X,t) + \gamma(\|u\|_{\UC})%
\end{equation*}
for all $(t,x,u) \in \R_+ \tm X \tm \UC$.%
\end{definition}

\begin{definition}
A function $V:X \rightarrow \R_+$ is called an \emph{ISS Lyapunov function (in an implication form)} for $\Sigma$ if it satisfies the following properties:%
\begin{enumerate}
\item[(i)] $V$ is continuous.%
\item[(ii)] There exist $\psi_1,\psi_2 \in \KC_{\infty}$ such that%
\begin{equation}\label{eq_isslf_coercivity}
  \psi_1(\|x\|_X) \leq V(x) \leq \psi_2(\|x\|_X) \mbox{\quad for all\ } x \in X.%
\end{equation}
\item[(iii)] There exist $\gamma\in\KC$ and $\alpha\in\PC$ such that for all $x\in X$ and $u\in\UC$ the following implication holds:%
\begin{equation}\label{eq_Lyap_impl}
  V(x) > \gamma(\|u\|_{\UC}) \quad \Rightarrow \quad \rmD^+ V_u(x) \leq -\alpha(V(x)),%
\end{equation}
where $\rmD^+ V_u(x)$ denotes the right upper Dini orbital derivative defined as%
\begin{equation*}
  \rmD^+ V_u(x) := \limsup_{t \rightarrow 0^+} \frac{V(\phi(t,x,u)) - V(x)}{t}.%
\end{equation*}
\end{enumerate}
\end{definition}

It is well-known that the existence of an ISS Lyapunov function implies ISS.%

\begin{theorem}\label{thm_smallgain_result}
Consider the network $\Sigma$ composed of the subsystems $\Sigma_i$, $i \in \N$, and let $\{ V_i : i \in \N \}$ be a family of associated ISS Lyapunov functions with associated MAFs $\mu_i$, internal gains $\gamma_{ij}$, and external gains $\gamma_{iu}$. Additionally, let the following assumptions hold:%
\begin{enumerate}
\item[(i)] The system $\Sigma$ is well-posed.%
\item[(ii)] The gain operator $\Gamma:\ell_{\infty}^+ \rightarrow \ell_{\infty}^+$ is well-defined and continuous and there exists $\gamma^u_{\max} \in \KC$ with $\gamma_{iu} \leq \gamma^u_{\max}$ for all $i \in \N$.%
\item[(iii)] There exists a path of strict decay for $\Gamma$.%
\item[(iv)] For each $R > 0$, there exists a constant $L > 0$ with%
\begin{equation*}
  |V_i(x_i) - V_i(y_i)| \leq L|x_i - y_i| \mbox{\quad for all\ } x_i,y_i \in B_R(0),\ i \in \N.%
\end{equation*}
\end{enumerate}
Then $\Sigma$ is ISS and the following function is an ISS Lyapunov function for $\Sigma$:%
\begin{equation*}
  V(x) = \sup_{i\in\N}\sigma_i^{-1}(V_i(x_i)) \mbox{\quad for all\ } x \in X.%
\end{equation*}
Moreover, $V$ is locally Lipschitz continuous on $X \setminus \{0\}$.
\end{theorem}

The proof of the theorem is almost identical to the one given in \cite[Thm.~III.1]{kawan2021lyapunov} for the case of max-type gain operators. The simple modifications necessary in Step 4 of the proof are left to the reader.%

The rest of the paper is devoted to the construction of a path of decay for the gain operator $\Gamma$. (Recall that by Remark \ref{rem_path_of_decay}, a path of strict decay for $\Gamma$ is obtained from a path of decay for the enlarged gain operator $\Gamma_{\rho}$.)%

\section{The augmented gain operator}\label{sec_augmented_op}

In this section, we introduce and study an operator, derived from the gain operator $\Gamma$, which will be our main tool for the construction of a path of decay for $\Gamma$.%

\subsection{Definition and elementary properties}

Given a well-defined and continuous gain operator $\Gamma$, we introduce the \emph{augmented gain operator}%
\begin{equation}\label{eq_def_hatgamma}
  \hat{\Gamma}(s) := s \oplus \Gamma(s),\quad \hat{\Gamma}:\ell_{\infty}^+ \rightarrow \ell_{\infty}^+.%
\end{equation}
Some elementary properties of $\hat{\Gamma}$ are provided in the next proposition.%

\begin{proposition}\label{prop_elem}
The operator $\hat{\Gamma}$ has the following properties:%
\begin{enumerate}
\item[(i)] $\hat{\Gamma}$ is well-defined and continuous.%
\item[(ii)] $\hat{\Gamma}(0) = 0$ and $\hat{\Gamma}$ is a monotone operator.%
\item[(iii)] $\hat{\Gamma}(s) \geq s$ for all $s \in \ell_{\infty}^+$.%
\item[(iv)] $\hat{\Gamma}(s) = s$ if and only if $\Gamma(s) \leq s$. That is, the fixed points of $\hat{\Gamma}$ are precisely the points of decay for $\Gamma$.%
\item[(v)] $\bigoplus_{i=0}^n \Gamma^i(s) \leq \hat{\Gamma}^n(s)$ for all $n \geq 0$ and $s \in \ell_{\infty}^+$.%
\item[(vi)] $\hat{\Gamma}^n(s) = s \oplus \Gamma(\hat{\Gamma}^{n-1}(s))$ for all $n \geq 1$ and $s \in \ell_{\infty}^+$.%
\end{enumerate}
\end{proposition}

\begin{proof}
(i)--(iv) are obvious or follow easily from the construction.%

(v) We prove the inequality by induction. For $n=0$, it is trivial. Assuming that $\bigoplus_{i=0}^n \Gamma^i(s) \leq \hat{\Gamma}^n(s)$ for a fixed $n$, we obtain%
\begin{align*}
  \hat{\Gamma}^{n+1}(s) &= \hat{\Gamma}^n(s) \oplus \Gamma(\hat{\Gamma}^n(s)) \geq \bigoplus_{i=0}^n \Gamma^i(s) \oplus \Gamma\Bigl(\bigoplus_{i=0}^n \Gamma^i(s)\Bigr) \\
	&\geq \bigoplus_{i=0}^n \Gamma^i(s) \oplus \Gamma(\Gamma^n(s)) = \bigoplus_{i=0}^{n+1} \Gamma^i(s).%
\end{align*}

(vi) We prove the formula by induction. For $n = 1$, it holds by the definition of $\hat{\Gamma}$. Assuming that it holds for a fixed $n$, we obtain%
\begin{align*}
  \hat{\Gamma}^{n+1}(s) &= \hat{\Gamma}(\hat{\Gamma}^n(s)) = \hat{\Gamma}^n(s) \oplus \Gamma(\hat{\Gamma}^n(s)) = s \oplus \Gamma(\hat{\Gamma}^{n-1}(s)) \oplus \Gamma(\hat{\Gamma}^n(s)).%
\end{align*}
Since $\hat{\Gamma}(s) \geq s$ and $\hat{\Gamma}$ is a monotone operator, it follows that $\hat{\Gamma}^n(s) \geq \hat{\Gamma}^{n-1}(s)$. By monotonicity of $\Gamma$, we obtain $\Gamma(\hat{\Gamma}^n(s)) \geq \Gamma(\hat{\Gamma}^{n-1}(s))$ and this implies $\hat{\Gamma}^{n+1}(s) = s \oplus \Gamma(\hat{\Gamma}^n(s))$.
\end{proof}

Another important property shared by $\Gamma$ and $\hat{\Gamma}$ is described in the next lemma.%

\begin{lemma}\label{lem_weak_continuity}
If $s^n \rightarrow s$ componentwise for a sequence $(s^n)_{n\in\N}$ in $\ell_{\infty}^+$, then $\Gamma(s^n) \rightarrow \Gamma(s)$ and $\hat{\Gamma}(s^n) \rightarrow \hat{\Gamma}(s)$ componentwise. In particular, if a trajectory of $\Gamma$ (or $\hat{\Gamma}$) converges componentwise to $s^* \in \ell_{\infty}^+$, then $s^*$ is a fixed point of $\Gamma$ (or $\hat{\Gamma}$).
\end{lemma}

\begin{proof}
Assume that $s^n \rightarrow s$ componentwise and fix $i \in \N$. Then%
\begin{align*}
  \lim_{n \rightarrow \infty} \Gamma_i(s^n) = \lim_{n \rightarrow \infty}\mu_i( [\gamma_{ij}(s^n_j)]_{j\in\N} ).%
\end{align*}
Observe that $(\gamma_{ij}(s^n_j))_{j\in\N} \in \ell_{\infty}^+(I_i)$ for all $n$. By property (iii) of MAFs, one obtains%
\begin{align*}
  \lim_{n \rightarrow \infty} \Gamma_i(s^n) &= \mu_i( \lim_{n \rightarrow \infty}[\gamma_{ij}(s^n_j)]_{j\in\N} ) \\
	&= \mu_i( [\gamma_{ij}(\lim_{n \rightarrow \infty} s^n_j)]_{j\in\N} ) = \mu_i( [\gamma_{ij}(s_j)]_{j\in\N} ) = \Gamma_i(s).%
\end{align*}
The proof for $\hat{\Gamma}$ is very similar and hence is omitted. As a consequence, if a trajectory $\Gamma^n(s)$ converges componentwise to $s^*$, then $\Gamma(\Gamma^n(s)) = \Gamma^{n+1}(s)$ converges componentwise to $\Gamma(s^*)$. This immediately implies $\Gamma(s^*) = s^*$ (analogously for $\hat{\Gamma}$ in place of $\Gamma$).
\end{proof}

\begin{remark}
We note that the property of the operators $\Gamma$ and $\hat{\Gamma}$ described in the above lemma is related to continuity with respect to the weak$^*$-topology of $\ell_{\infty}$ (as the dual of $\ell_1$), because a sequence converges in this topology if and only if it is norm-bounded and converges componentwise. However, since we will not explicitly use the weak$^*$-topology, we do not turn this into a formal statement.
\end{remark}

Our approach to the construction of a path of decay is based on stability properties of the two discrete-time systems%
\begin{equation}\label{eq_gainop_sys}
  s(k+1) = \Gamma(s(k)),\quad s(0) \in \ell_{\infty}^+%
\end{equation}
and%
\begin{equation}\label{eq_auggainop_sys}
  s(k+1) = \hat{\Gamma}(s(k)),\quad s(0) \in \ell_{\infty}^+.%
\end{equation}

As an easy corollary from Proposition \ref{prop_elem}(v), we obtain the following result.%

\begin{corollary}\label{cor_ugs}
If the system \eqref{eq_auggainop_sys} induced by $\hat{\Gamma}$ is UGS, then so is the system \eqref{eq_gainop_sys} induced by $\Gamma$.
\end{corollary}

The assumption that the system induced by $\hat{\Gamma}$ is UGS immediately implies the existence of plenty of points of decay for $\Gamma$, as shown in the next proposition.%

\begin{proposition}\label{prop_pointsofdecay}
Assume that the system \eqref{eq_auggainop_sys} induced by $\hat{\Gamma}$ is UGS, i.e.~there exists $\varphi \in \KC_{\infty}$ such that%
\begin{equation*}
  \|\hat{\Gamma}^k(s)\|_{\infty} \leq \varphi(\|s\|_{\infty}) \mbox{\quad for all\ } s \in \ell_{\infty}^+,\ k \in \Z_+.%
\end{equation*}
Then, for each $s \in \ell_{\infty}^+$, there is $s^* \in \ell_{\infty}^+$ such that $\Gamma(s^*) \leq s^*$ and $s \leq s^* \leq \varphi(\|s\|_{\infty}) \unit$.
\end{proposition}

\begin{proof}
Observe that statements (ii) and (iii) in Proposition \ref{prop_elem} together imply%
\begin{equation*}
  \hat{\Gamma}^{k+1}(s) \geq \hat{\Gamma}^k(s) \mbox{\quad for all\ } s \in \ell_{\infty}^+,\ k \in \Z_+.%
\end{equation*}
Since $\|\hat{\Gamma}^k(s)\|_{\infty} \leq \varphi(\|s\|_{\infty})$ for all $k \in \Z_+$ by assumption, each component sequence $(\hat{\Gamma}^k_i(s))_{k\in\Z_+}$ converges to a number $s^*_i \leq \varphi(\|s\|_{\infty})$ as $k \rightarrow \infty$ and the vector $s^* := (s^*_i)_{i\in\N}$ is an element of $\ell_{\infty}^+$. From Lemma \ref{lem_weak_continuity}, it immediately follows that $s^*$ is a fixed point of $\hat{\Gamma}$, and thus by Proposition \ref{prop_elem}(iv), a point of decay for $\Gamma$. Since $s \leq s^* \leq \varphi(\|s\|_{\infty})\unit$ by construction, the proof is complete.
\end{proof}

Under the assumption that $\hat{\Gamma}$ induces a UGS system, we introduce the operator%
\begin{equation*}
  \hat{Q}(s) := \bigoplus_{k=0}^{\infty}\hat{\Gamma}^k(s),\quad \hat{Q}:\ell_{\infty}^+ \rightarrow \ell_{\infty}^+,%
\end{equation*}
and observe that in Proposition \ref{prop_pointsofdecay} we have $s^* = \hat{Q}(s)$.%

\subsection{Uniform global stability for particular types of gain operators}

In this subsection, we recover some known results for particular types of gain operators by showing that $\hat{\Gamma}$ induces a UGS system under assumptions previously used to construct points of decay by other methods.%

We start with max-type gain operators.%

\begin{proposition}\label{prop_maxop1}
Let $\Gamma$ be a max-type gain operator. Then the system \eqref{eq_auggainop_sys} induced by $\hat{\Gamma}$ is UGS if and only if the system \eqref{eq_gainop_sys} induced by $\Gamma$ is UGS. In this case,%
\begin{equation}\label{eq_maxtype_hatq}
  \hat{Q}(s) = \bigoplus_{k=0}^{\infty}\Gamma^k(s) \mbox{\quad for all\ } s \in \ell_{\infty}^+.%
\end{equation}
\end{proposition}

\begin{proof}
Using the fact that max-type gain operators are max-preserving together with Proposition \ref{prop_elem}(vi), we obtain%
\begin{equation*}
  \hat{\Gamma}^{n+1}(s) = s \oplus \Gamma(\hat{\Gamma}^n(s)) = s \oplus \Gamma(s) \oplus \Gamma^2(\hat{\Gamma}^{n-1}(s)).%
\end{equation*}
Proceeding inductively, this yields%
\begin{equation}\label{eq_maxtype_dyn_rel}
  \hat{\Gamma}^n(s) = \bigoplus_{k=0}^n \Gamma^k(s) \mbox{\quad for all\ } n \geq 0.%
\end{equation}
This formula easily implies both statements of the proposition.
\end{proof}

Proposition \ref{prop_maxop1} shows that the operator $\hat{Q}$ coincides with the operator $Q$, introduced in \cite{kawan2021lyapunov} for the construction of paths of strict decay in the case of max-type gain operators.%

Now, let us look at the class of subadditive and homogeneous gain operators.%

\begin{proposition}\label{prop_homogeneous1}
Assume that $\Gamma$ is a subadditive and homogenenous gain operator satisfying%
\begin{equation}\label{eq_linear_cond}
  \inf_{n \in \N} \|\Gamma^n(\unit)\|_{\infty} < 1.%
\end{equation}
Then the system \eqref{eq_auggainop_sys} induced by $\hat{\Gamma}$ is UGS.
\end{proposition}

\begin{proof}
We first prove by induction that%
\begin{equation}\label{eq_linear_sum_ugs}
  \hat{\Gamma}^n(\unit) \leq \sum_{k=0}^n \Gamma^k(\unit) \mbox{\quad for all\ } n \geq 0.%
\end{equation}
For $n = 0$, both sides of \eqref{eq_linear_sum_ugs} are equal to $\unit$. Now, assume that the statement holds for a fixed $n$. Then, using the assumptions on the $\mu_i$ and $\gamma_{ij}$, we obtain%
\begin{align*}
  \hat{\Gamma}^{n+1}_i(\unit) &= \hat{\Gamma}_i(\hat{\Gamma}^n(\unit)) = \max\{ \hat{\Gamma}^n_i(\unit), \Gamma_i(\hat{\Gamma}^n(\unit)) \} \leq \max\Bigl\{ \sum_{k=0}^n \Gamma^k_i(\unit), \Gamma_i\Bigl( \sum_{k=0}^n \Gamma^k(\unit) \Bigr) \Bigr\} \\
	&\leq \max\Bigl\{ \sum_{k=0}^n \Gamma^k_i(\unit), \sum_{k=0}^n \Gamma_i(\Gamma^k(\unit)) \Bigr\} = \max\Bigl\{ \sum_{k=0}^n \Gamma^k_i(\unit), \sum_{k=0}^n \Gamma_i^{k+1}(\unit)) \Bigr\} \\
	&\leq \max\Bigl\{ \sum_{k=0}^n \Gamma^k_i(\unit) + \Gamma^{n+1}_i(\unit), \unit + \sum_{k=1}^{n+1} \Gamma_i^k(\unit) \Bigr\} = \sum_{k=0}^{n+1} \Gamma^k_i(\unit).%
\end{align*}
This proves \eqref{eq_linear_sum_ugs}. In \cite[Prop.~9]{mironchenko2021iss}, it was shown that \eqref{eq_linear_cond} implies UGES of the system induced by $\Gamma$, i.e.~$\|\Gamma^k(s)\|_{\infty} \leq M\alpha^k\|s\|_{\infty}$ for all $s \in \ell_{\infty}^+$, $k \in \Z_+$ with constants $M > 0$ and $\alpha \in (0,1)$. This implies UGS of system \eqref{eq_auggainop_sys}. Indeed, for any $s \in \ell_{\infty}^+$,%
\begin{align*}
  \|\hat{\Gamma}^n(s)\|_{\infty} &\leq \|\hat{\Gamma}^n(\|s\|_{\infty}\unit)\|_{\infty} \leq \|s\|_{\infty} \sum_{k=0}^n  \|\Gamma^k(\unit)\|_{\infty} \\
	&\leq \|s\|_{\infty} \sum_{k=0}^{\infty} \|\Gamma^k(\unit)\|_{\infty} \leq \|s\|_{\infty} \sum_{k=0}^{\infty} M\alpha^k = \frac{M}{1-\alpha} \|s\|_{\infty}.%
\end{align*}
The proof is complete.
\end{proof}

The above proposition partially recovers \cite[Prop.~9]{mironchenko2021iss}, which shows that condition \eqref{eq_linear_cond} even guarantees the existence of a linear path of strict decay for $\Gamma$.%

In the case of a linear gain operator, condition \eqref{eq_linear_cond} is equivalent to the spectral radius condition $r(\Gamma) < 1$. Hence, we obtain the following corollary.%

\begin{corollary}
Let $\Gamma$ be a linear gain operator with $r(\Gamma) < 1$. Then the system induced by $\hat{\Gamma}$ is UGS.
\end{corollary}

\subsection{Stability and small-gain conditions}

In this subsection, we relate the uniform global stability of the system induced by $\hat{\Gamma}$ to small-gain conditions previously studied in the literature on the small-gain approach for both finite and infinite networks.%

\begin{definition}
We say that the gain operator $\Gamma$ satisfies the%
\begin{itemize}
\item \emph{small-gain condition (SGC)} if $\Gamma(s) \not\geq s$ for all $s \in \ell_{\infty}^+ \setminus \{0\}$.%
\item \emph{uniform small-gain condition} if there exists $\eta \in \KC_{\infty}$ such that%
\begin{equation*}
  \dist(\Gamma(s) - s,\ell_{\infty}^+) \geq \eta(\|s\|_{\infty}) \mbox{\quad for all\ } s \in \ell_{\infty}^+.%
\end{equation*}
\item \emph{monotone bounded invertibility (MBI) property} if there exists $\xi \in \KC_{\infty}$ such that for all $s,b \in \ell_{\infty}^+$ the following implication holds:%
\begin{equation}\label{eq_mbi}
  (\id - \Gamma)(s) \leq b \quad \Rightarrow \quad \|s\|_{\infty} \leq \xi(\|b\|_{\infty}).%
\end{equation}
\item \emph{$\oplus$-MBI property} if there exists $\varphi \in \KC_{\infty}$ such that for all $s,b \in \ell_{\infty}^+$ the following implication holds:%
\begin{equation*}
  s \leq \Gamma(s) \oplus b \quad \Rightarrow \quad \|s\|_{\infty} \leq \varphi(\|b\|_{\infty}).%
\end{equation*}
\end{itemize}
\end{definition}

\begin{remark}
The uniform small-gain condition and the monotone bounded invertibility property were introduced (and shown to be equivalent) in \cite{mironchenko2021nonlinear}. The name ``monotone bounded invertibility property'' actually refers to the case when the gain operator is linear. Then the validity of the implication \eqref{eq_mbi} implies that $\id - \Gamma$ is invertible with a bounded inverse, cf.~\cite[Thm.~3.3]{gluck2021stability}.%
\end{remark}

The following proposition is the main result of this subsection.%

\begin{proposition}\label{prop_smallgain_conds}
Consider the following statements:%
\begin{enumerate}
\item[(a)] The operator $\Gamma$ satisfies the uniform SGC.%
\item[(b)] The operator $\Gamma$ satisfies the MBI property.%
\item[(c)] The operator $\Gamma$ satisfies the $\oplus$-MBI property.%
\item[(d)] The system \eqref{eq_auggainop_sys} induced by $\hat{\Gamma}$ is UGS and the system \eqref{eq_gainop_sys} induced by $\Gamma$ is globally componentwise attractive.%
\item[(e)] The system \eqref{eq_auggainop_sys} induced by $\hat{\Gamma}$ is UGS and $\Gamma$ satisfies the SGC.%
\item[(f)] The system \eqref{eq_auggainop_sys} induced by $\hat{\Gamma}$ is UGS and $\Gamma$ has no non-zero fixed points.%
\end{enumerate}
Then%
\begin{equation*}
  \mbox{(a)} \Leftrightarrow \mbox{(b)} \begin{array}{c} \Rightarrow \\ \nLeftarrow \end{array} \mbox{(c)} \Rightarrow  \mbox{(d)} \Leftrightarrow \mbox{(e)} \Leftrightarrow \mbox{(f)}.%
\end{equation*}
\end{proposition}

\begin{proof}
The equivalence ``(a) $\Leftrightarrow$ (b)'' has been proven in \cite[Prop.~7.1]{mironchenko2021nonlinear}.%

``(b) $\Rightarrow$ (c)'': Assume that $s \leq \Gamma(s) \oplus b$ for some $s,b \in \ell_{\infty}^+$. Since the maximum of two numbers is bounded above by their sum, this implies $s \leq \Gamma(s) + b$, or equivalently, $(\id - \Gamma)(s) \leq b$. Hence, if (b) holds with some $\xi \in \KC_{\infty}$, then (c) holds with $\varphi = \xi$.%
 
``(c) $\Rightarrow$ (d)'': By Proposition \ref{prop_elem}(vi), we have $\hat{\Gamma}^n(s) = s \oplus \Gamma(\hat{\Gamma}^{n-1}(s))$ for all $n \geq 1$, $s \in \ell_{\infty}^+$. This implies $\hat{\Gamma}^n(s) \leq \hat{\Gamma}^{n+1}(s) = \Gamma(\hat{\Gamma}^n(s)) \oplus s$. Hence, by assumption, $\|\hat{\Gamma}^n(s)\|_{\infty} \leq \varphi(\|s\|_{\infty})$. Since this holds for every $n$, the system induced by $\hat{\Gamma}$ is UGS. Consequently, for every $s \in \ell_{\infty}^+$, the operator%
\begin{equation*}
  \hat{Q}(s) = \bigoplus_{k=0}^{\infty}\hat{\Gamma}^k(s)%
\end{equation*}
is well-defined and $\Gamma(\hat{Q}(s)) \leq \hat{Q}(s)$. Since $\Gamma$ is a monotone operator, this implies $\Gamma^n(\hat{Q}(s)) \leq \Gamma^{n-1}(\hat{Q}(s))$ for all $n\geq 1$. Hence, componentwise, the sequence $s^n := \Gamma^n(\hat{Q}(s))$ is monotonically decreasing and bounded below by zero. Therefore, each component sequence converges to a nonnegative real number. Putting all these numbers together in a vector $\hat{s}$, Lemma \ref{lem_weak_continuity} shows that $\Gamma(\hat{s}) = \hat{s}$. This implies $\hat{s} \leq \Gamma(\hat{s}) \oplus 0$, and hence (c) yields $\hat{s} = 0$. We thus observed that the trajectory $(\Gamma^n(\hat{Q}(s)))_{n\in\N}$ converges to zero componentwise. Since $\Gamma^n(s) \leq \Gamma^n(\hat{Q}(s))$, the same is true for the trajectory of $s$.%

``(c) $\nRightarrow$ (b)'': \cite[Ex.~6.3 and Thm.~6.4]{ruffer2010monotone} shows that (b) and (c) are \emph{not} equivalent for max-type gain operators.%

``(d) $\Rightarrow$ (e)'': Assume that $s \leq \Gamma(s)$ for some $s \in \ell_{\infty}^+$. This implies $\Gamma^{n-1}(s) \leq \Gamma^n(s)$ for all $n \geq 1$. Hence, each component sequence $(\Gamma^n_i(s))_{n\in\N}$ is monotonically non-decreasing, but by assumption converges to zero. Hence, $s = 0$, which shows that $\Gamma$ satisfies the SGC.%

``(e) $\Rightarrow$ (d)'': For each $s \in \ell_{\infty}^+$, consider $s^* := \hat{Q}(s)$. We know that $\Gamma(s^*) \leq s^*$, implying that each component sequence $(\Gamma^n_i(s^*))_{n\in\N}$ is monotonically decreasing. From Lemma \ref{lem_weak_continuity}, it follows that the componentwise limit is a fixed point of $\Gamma$. Since $\Gamma$ satisfies the SGC by assumption, $s=0$ is the only fixed point.%

``(e) $\Rightarrow$ (f)'': This is trivial.%

``(f) $\Rightarrow$ (e)'': Assume that $\Gamma(s) \geq s$ for some $s \in \ell_{\infty}^+$. This implies $\hat{\Gamma}(s) = s \oplus \Gamma(s) = \Gamma(s)$, $\hat{\Gamma}^2(s) = s \oplus \Gamma(\hat{\Gamma}(s)) = s \oplus \Gamma^2(s) = \Gamma^2(s)$. Indeed, by induction one can show that $\hat{\Gamma}^n(s) = \Gamma^n(s)$ for all $n \geq 0$. Since the system induced by $\hat{\Gamma}$ is UGS, it follows that the trajectory $(\Gamma^n(s))_{n\in\N}$ converges componentwise to some $s^* \in \ell_{\infty}^+$. By Lemma \ref{lem_weak_continuity}, we have $\Gamma(s^*) = s^*$. By assumption, $s^* = 0$ and thus $s = 0$, showing that $\Gamma$ satisfies the SGC.
\end{proof}

It is an open question whether (c) and (d) are equivalent. In the following proposition, we show that they are equivalent for max-type gain operators.%

\begin{proposition}\label{prop_hatgamma_ugs}
For a max-type gain operator $\Gamma$, consider the following statements:%
\begin{enumerate}
\item[(a)] The system \eqref{eq_gainop_sys} induced by $\Gamma$ is UGS and globally componentwise attractive.%
\item[(b)] The operator $\Gamma$ satisfies the max-robust SGC:\footnote{This property was introduced in \cite{kawan2021lyapunov}.} there exists $\omega \in \KC_{\infty}$, $\omega < \id$, such that for all $i,j\in\N$%
\begin{equation*}
  \Gamma(s) \oplus \omega(s_j)e_i \not\geq s \mbox{\quad for all\ } s \in \ell_{\infty}^+ \setminus \{0\}.%
\end{equation*}
\item[(c)] The operator $\Gamma$ satisfies the $\oplus$-MBI property.%
\item[(d)] The system \eqref{eq_auggainop_sys} induced by $\hat{\Gamma}$ is UGS and the system \eqref{eq_gainop_sys} induced by $\Gamma$ is globally componentwise attractive.%
\item[(e)] The system \eqref{eq_auggainop_sys} induced by $\hat{\Gamma}$ is UGS and each of its trajectories is componentwise eventually constant.%
\end{enumerate}
Then (a) $\Leftrightarrow$ (b) $\Leftrightarrow$ (c) $\Leftrightarrow$ (d) $\begin{array}{c} \Rightarrow \\ \nLeftarrow \end{array}$ (e). \end{proposition}

\begin{proof}
``(a) $\Leftrightarrow$ (b)'': This was shown in \cite[Prop.~V.2]{kawan2021lyapunov}.%

``(a) $\Leftrightarrow$ (d)'': This follows from Proposition \ref{prop_maxop1}.%

``(c) $\Rightarrow$ (b)'':  Assume that condition (c) holds for some $\varphi \in \KC_{\infty}$ and that for some $i,j \in \N$ and $s \neq 0$ we have $\Gamma(s) \oplus \omega(s_j)e_i \geq s$, where $\omega < \varphi^{-1}$. Then%
\begin{equation*}
  \|s\|_{\infty} \leq \varphi(\omega(s_j)) < s_j \leq \|s\|_{\infty},%
\end{equation*}
which is a contradiction. Hence, (b) holds for any $\omega < \varphi^{-1}$.%

``(a) $\Rightarrow$ (c)'': This was shown in the proof of \cite[Prop.~V.2]{kawan2021lyapunov}.%

``(d) $\Rightarrow$ (e)'': By Proposition \ref{prop_maxop1}, the system induced by $\Gamma$ is UGS if and only if the one induced by $\hat{\Gamma}$ is UGS. By \eqref{eq_maxtype_dyn_rel}, the dynamics of $\Gamma$ and of $\hat{\Gamma}$ are related by%
\begin{equation}\label{eq_gamma_hatgamma}
  \hat{\Gamma}^k_i(s) = \max\{s_i,\Gamma_i(s),\Gamma^2_i(s),\ldots,\Gamma^k_i(s)\}.%
\end{equation}
Hence, the componentwise convergence of $\Gamma^k(s)$ to zero implies that for every $s$ there exists $k_0$ such that%
\begin{equation*}
  \hat{\Gamma}^k_i(s) = \hat{\Gamma}^{k_0}_i(s) \mbox{\quad for all\ } k \geq k_0.%
\end{equation*}

``(e) $\nRightarrow$ (d)'': Consider the interconnection of two systems with gains $\gamma_{12} = \gamma_{21} = \id$ and max-type gain operator. It then follows that $\hat{\Gamma}(s) = \max\{s_1,s_2\}\unit$ and $\hat{\Gamma}^k(s) = \hat{\Gamma}(s)$ for all $k \geq 2$. Hence, the system induced by $\hat{\Gamma}$ is UGS and each of its trajectories is eventually constant. However, the system induced by $\Gamma$ is obviously not globally componentwise attractive.
\end{proof}

For sum-type operators, we can provide the following sufficient condition for the $\oplus$-MBI property.%

\begin{proposition}
Assume that $\Gamma$ is a sum-type operator and the following assumptions hold:%
\begin{enumerate}
\item[(i)] $\Gamma^2$ satisfies the MBI property.%
\item[(ii)] There exists some $\eta \in \KC_{\infty}$ with $\gamma_{ij} \leq \eta$ for all $i,j \in \N$.%
\item[(iii)] There exists a bound on the cardinality of $I_i$.%
\end{enumerate}
Then $\Gamma$ satisfies the $\oplus$-MBI property.
\end{proposition}

\begin{proof}
Assume that $s \leq \Gamma(s) \oplus b$ for some $s,b \in \ell_{\infty}^+$. This implies $s \leq \Gamma(\Gamma(s) \oplus b) \oplus b$. Now, observe that%
\begin{align*}
  \Gamma_i(\Gamma(s) \oplus b) &= \sum_{j \in I_i} \gamma_{ij}(\max\{\Gamma_j(s),b_j\}) = \sum_{j \in I_i} \max\{\gamma_{ij}(\Gamma_j(s)),\gamma_{ij}(b_j)\} \\
	&\leq \sum_{j \in I_i} \gamma_{ij}(\Gamma_j(s)) + \sum_{j \in I_i} \gamma_{ij}(b_j) = \Gamma_i^2(s) + \Gamma_i(b).%
\end{align*}
Hence, $s \leq (\Gamma^2(s) + \Gamma(b)) \oplus b \leq \Gamma^2(s) + (b + \Gamma(b))$. By assumption, this implies $\|s\|_{\infty} \leq \varphi( \|b + \Gamma(b)\|_{\infty} )$ for some $\varphi \in \KC_{\infty}$. Therefore,%
\begin{equation*}
  \|s\|_{\infty} \leq \varphi( \|b\|_{\infty} + \|\Gamma(b)\|_{\infty} ) \leq \varphi(2\|b\|_{\infty}) + \varphi(2\|\Gamma(b)\|_{\infty}).%
\end{equation*}
Now, observe that $b \leq \|b\|_{\infty}\unit$, and hence our assumptions imply%
\begin{align*}
  \|\Gamma(b)\|_{\infty} \leq \|\Gamma(\|b\|_{\infty}\unit)\|_{\infty} = \sup_{i\in\N}\sum_{j\in I_i}\gamma_{ij}(\|b\|_{\infty}) \leq (\sup_i\# I_i) \eta(\|b\|_{\infty}) =: \psi(\|b\|_{\infty}).%
\end{align*}
Altogether, $\|s\|_{\infty} \leq \varphi(2\|b\|_{\infty}) + \varphi(2 \psi(\|b\|_{\infty}))$, which completes the proof.
\end{proof}

\section{Uniform global asymptotic stability}\label{sec_ugas}

In this section, our goal is to characterize uniform global asymptotic stability of the system \eqref{eq_gainop_sys} induced by $\Gamma$ in terms of the behavior of $\Gamma$ in one iterate. Our motivation is the result \cite[Thm.~VI.1]{kawan2021lyapunov}, which shows that in the max-type case UGAS together with some mild regularity assumptions guarantees the existence of a path of decay.%

We need the following definition.%

\begin{definition}
For any $i,j \in \N$ and $k\in\Z_+$, we say that subsystem $\Sigma_j$ influences subsystem $\Sigma_i$ in $k$ steps of time if there exist indices $j_1,j_2,\ldots,j_{k+1} \in \N$ with $j = j_1$ and $i = j_{k+1}$ such that $j_l \in I_{j_{l+1}}$ for $l = 1,\ldots,k$. We write $\NC^-_i(n)$ for the set of all indices $j$ such that $\Sigma_j$ influences $\Sigma_i$ in $k \in \{0,1,\ldots,n-1\}$ steps of time. We write $\NC^+_i(n)$ for the set of all indices $j$ such that $\Sigma_j$ is influenced by $\Sigma_i$ in $k \in \{0,1,\ldots,n-1\}$ steps of time.  
\end{definition}

Observe that in the particular case $k = 0$, the definition says that each subsystem $\Sigma_i$ influences itself in $0$ steps of time, since the condition reduces to $i = j$. In particular, observe that the sets $\NC^-_i(n)$ and $\NC^+_i(n)$ are never empty.%

The following lemma is the basis of our investigations in this section.%

\begin{lemma}\label{lem_hatgg}
Let $s \in \ell_{\infty}^+$, $i \in \N$ and $k \geq 1$. Assume that $\Gamma_j(s) \geq s_j$ for all $j \in \NC^-_i(k)$. Then%
\begin{equation}\label{eq_ghge}
  \hat{\Gamma}^k_i(s) = \Gamma^k_i(s).%
\end{equation}
\end{lemma}

\begin{proof}
We prove the statement by induction over $k$ (while $s$ is fixed and $i$ is variable). First, we prove a simpler statement under the same hypothesis, namely $\Gamma^k_i(s) \geq s_i$ (instead of \eqref{eq_ghge}). For $k = 1$, this holds because $\Gamma_i(s_i) \geq s_i$ by the hypothesis. Assume that it holds for a fixed $k$. Observe that $\Gamma_i^{k+1}(s) = \Gamma_i(\Gamma^k(s))$ and this only depends on $\Gamma^k_j(s)$ with $j \in I_i$. For these $j$, we know that $\Gamma^k_j(s) \geq s_j$ by induction hypothesis, and by monotonicity $\Gamma_i^{k+1}(s) = \Gamma_i(\Gamma^k(s)) \geq \Gamma_i(s) \geq s_i$.%

Now, we prove the main statement. For $k = 1$, our assumption reads $\Gamma_i(s) \geq s_i$. This implies $\hat{\Gamma}_i(s) = \max\{s_i,\Gamma_i(s)\} = \Gamma_i(s)$. Assume that the statement is true for a fixed $k$ and consider $k+1$. By Proposition \ref{prop_elem}(vi), we have $\hat{\Gamma}^{k+1}_i(s) = \max\{s_i,\Gamma_i(\hat{\Gamma}^k(s))\}$. Now, $\Gamma_i(\hat{\Gamma}^k(s))$ only depends on $\hat{\Gamma}^k_j(s)$ for $j \in I_i$. Since we assume that $\Gamma_a(s) \geq s_a$ for all $a\in\NC^-_i(k+1)$, we also have $\Gamma_a(s) \geq s_a$ for all $a \in \NC^-_j(k)$ whenever $j \in I_i$. By the induction hypothesis, this implies $\hat{\Gamma}^k_j(s) = \Gamma^k_j(s)$ for all $j \in I_i$, and hence%
\begin{equation*}
  \hat{\Gamma}^{k+1}_i(s) = \max\{s_i,\Gamma_i(\hat{\Gamma}^k(s))\} = \max\{s_i,\Gamma_i^{k+1}(s)\} = \Gamma_i^{k+1}(s),%
\end{equation*}
where we use that $\Gamma_i^{k+1}(s) \geq s_i$, as proved previously.
\end{proof}

The next proposition shows that UGAS implies something much stronger than the small-gain condition. For later purposes, here we consider a scaled gain operator.%

\begin{proposition}\label{prop_ugas_char_part1}
Consider the following statements:%
\begin{enumerate}
\item[(a)] There exists $\omega \in \KC_{\infty}$ with $\omega \leq \id$ such that the system induced by the operator $\Gamma_{\omega} := \omega^{-1} \circ \Gamma$ is UGAS.%
\item[(b)] For all $\alpha,r > 0$, there exists $n = n(r,\alpha) \in \N$ such that for all $s \in \ell_{\infty}^+$ and $i \in \N$ the following implication holds:%
\begin{equation*}
  \|s\|_{\infty} \leq r \mbox{\quad and \quad} s_i \geq \alpha \quad \Rightarrow \quad \Gamma_j(s) < \omega(s_j) \mbox{\ for some\ } j \in \NC^-_i(n).%
\end{equation*}
\end{enumerate}
Then (a) $\Rightarrow$ (b).
\end{proposition}

\begin{proof}
Assume that (a) holds, but (b) does not. Hence, there exist $\alpha>0$ and $r > 0$ such that for all $n \in \N$ there are $s^n \in \ell_{\infty}^+$ with $\|s^n\|_{\infty} \leq r$ and $i_n \in \N$ with $s^n_{i_n} \geq \alpha$ and $\Gamma_j(s^n) \geq \omega(s_j^n)$ for all $j \in \NC^-_{i_n}(n)$. By Lemma \ref{lem_hatgg}, this implies $\Gamma_{\omega,i_n}^n(s^n) \geq s^n_{i_n} \geq \alpha$. Our assumption implies that there exist $\beta \in \KC\LC$ and $m \in \N$ with%
\begin{equation*}
  \|\Gamma_{\omega}^m(s^m)\|_{\infty} \leq \beta(\|s^m\|_{\infty},m) \leq \beta(r,m) < \alpha.%
\end{equation*}
Hence, we obtain the following contradiction:%
\begin{equation*}
  \alpha \leq s^m_{i_m} \leq \Gamma_{\omega,i_m}^m(s^m) \leq \|\Gamma_{\omega}^m(s^m)\|_{\infty} \leq \beta(\|s^m\|_{\infty},m) < \alpha.%
\end{equation*}
The proof is complete.
\end{proof}

We would like to show that also the converse implication holds in Proposition \ref{prop_ugas_char_part1}. Surprisingly, it seems that the converse only holds if we replace $\NC^-_i(n)$ in statement (b) with $\NC^+_i(n)$ and if we choose the $\omega$ in (b) larger than the one in (a). Additionally, we need the system induced by $\hat{\Gamma}$ to be UGS and some mild technical assumptions which are introduced next.%

\begin{definition}
We say that the network satisfies \emph{standard technical assumptions} if the following holds:%
\begin{enumerate}
\item[(a)] There exists $\xi \in \KC_{\infty}$ such that $\Gamma_i(s) \geq \xi \circ \gamma_{ij}(s_j)$ for all $s\in\ell_{\infty}^+$, $i \in \N$ and $j \in I_i$.%
\item[(b)] There exists $\eta \in \KC$ with $\eta < \id$ such that $\gamma_{ij} \geq \eta$ for all $i\in\N$ and $j \in I_i$.%
\end{enumerate}
\end{definition}

Observe that Assumption (a) above is satisfied for most natural choices of MAFs. Assumption (b) is satisfied for any finite network. For infinite networks, it guarantees that the non-zero couplings between subsystems cannot be arbitrarily weak.%

\begin{proposition}\label{prop_ugas_char_part2}
Let the following assumptions hold:%
\begin{enumerate}
\item[(i)] The system induced by $\hat{\Gamma}$ is UGS.%
\item[(ii)] For all $\alpha,r > 0$, there exists $n = n(r,\alpha) \in \N$ such that $\|s\|_{\infty} \leq r$ and $s_i \geq \alpha$ imply $\Gamma_j(s) < \omega(s_j)$ for some $j \in \NC^+_i(n)$.%
\item[(iii)] The network satisfies standard technical assumptions.%
\item[(iv)] The number of subsystems that are influenced by a given subsystem $\Sigma_i$ in one step of time is bounded over $i \in \N$.%
\end{enumerate}
Then the system induced by $\Gamma$ is UGAS.%
\end{proposition}

\begin{proof}
From Assumption (i), it follows by Corollary \ref{cor_ugs} that the system induced by $\Gamma$ is UGS. Hence, it suffices to prove global attractivity (see, e.g., \cite[Thm.~4.2]{mironchenko2017uniform}). To this end, we pick $s \in \ell_{\infty}^+$ and show that $\|\Gamma^n(s)\|_{\infty} \rightarrow 0$ as $n \rightarrow \infty$. As $s \leq \hat{Q}(s)$ and the latter is well-defined by Assumption (i) and Proposition \ref{prop_pointsofdecay}, it suffices to prove that $\|\Gamma^n(\hat{Q}(s))\|_{\infty} \rightarrow 0$ as $n \rightarrow \infty$. Hence, we may assume without loss of generality that $\Gamma(s) \leq s$.%

We then have to show that for every $\alpha > 0$ there exists $n \in \N$ with $\Gamma^n_i(s) < \alpha$ for all $i \in \N$. Given some $\alpha > 0$, let us define%
\begin{equation*}
  I_{\alpha} := \{ i \in \N : s_i \geq \alpha \}.%
\end{equation*}
For all $i \in \N \setminus I_{\alpha}$, we have $\Gamma^n_i(s) \leq s_i < \alpha$ for all $n \in \N$ (following from $\Gamma(s) \leq s$ and monotonicity of $\Gamma$). Hence, we only need to show that%
\begin{equation*}
  \exists n \in \N\ \forall i \in I_{\alpha}:\ \Gamma^n_i(s) < \alpha.%
\end{equation*}
From Assumption (iv), it follows that the cardinality of $\NC^+_i(n)$ is bounded with respect to $i$, say $\# \NC^+_i(n) \leq B_n$ with a number $B_n$, only depending on $n$.%

Put $s^n := \Gamma^n(s)$ and observe that $\|s_n\|_{\infty} \leq \|s\|_{\infty} =: r$ for all $n \geq 0$ because $\Gamma(s) \leq s$. By Assumption (ii), there exists $m = m(r,\alpha)$ such that $\Gamma_j(s^n) < \omega(s_j^n)$ for some $j \in \NC^+_i(m)$ whenever $s^n_i \geq \alpha$.%

For each $i \in I_{\alpha}$, define%
\begin{equation*}
  K_i := \max\{ k \in \N : s^k_i \geq \alpha \}.%
\end{equation*}
Since Assumption (ii) implies that $\Gamma$ satisfies the SGC, it follows by Proposition \ref{prop_smallgain_conds} that every $\Gamma$-trajectory converges to $0$ componentwise, and hence $K_i$ is well-defined. We assume without loss of generality that $K_i \geq m$ for each $i \in I_{\alpha}$. In fact, we can neglect all $i$ with $K_i < m$.%

Then, for each $i \in I_{\alpha}$, there must exist some $j_i \in \NC^+_i(m)$ such that%
\begin{equation*}
  \# \{ k \in [0;K_i-1] : \Gamma_{j_i}(s^k) < \omega(s^k_{j_i}) \} \geq \frac{K_i}{\# \NC^+_i(m)}.%
\end{equation*}
If this was not the case, it would result in%
\begin{equation*}
  K_i = \# \{ k \in [0;K_i-1] : \Gamma_j(s^k) < \omega(s^k_j) \mbox{\ for some } j \in \NC^+_i(m) \} < K_i,%
\end{equation*}
which is a contradiction. Since at any instance $k \in [0;K_i-1]$, we have $\Gamma_{j_i}(s^k) \leq s^k_{j_i}$, this implies%
\begin{equation*}
  \Gamma^{K_i}_{j_i}(s) < \omega^{\lfloor K_i / \# \NC^+_i(m) \rfloor}(s_{j_i}) \leq \omega^{\lfloor K_i / B_m \rfloor}(r).%
\end{equation*}
By Assumption (iii), we have $\Gamma_i(s) \geq \xi \circ \gamma_{ij}(s_j)$ for all $j \in I_i$ and some $\xi \in \KC_{\infty}$. Now, assume that $i \in I_{k_1},k_1 \in I_{k_2},\ldots,k_{l-1} \in I_{j_i}$ with $l \leq m - 1$. Then%
\begin{align*}
  \Gamma^{K_i - l + 1}_{k_1}(s) = \Gamma_{k_1}( [\Gamma^{K_i - l}_j(s)]_{j \in I_{k_1}} ) \geq \xi \circ \gamma_{k_1 i}(\Gamma^{K_i-l}_i(s)) \geq \xi \circ \gamma_{k_1 i}(\alpha).%
\end{align*}
This, in turn, implies%
\begin{align*}
  \Gamma^{K_i - l + 2}_{k_2}(s) &= \Gamma_{k_2}( [\Gamma^{K_i - l + 1}(s)]_{j \in I_{k_2}} ) \\
	&\geq \xi \circ \gamma_{k_2 k_1}(\Gamma^{K_i - l + 1}_{k_1}(s)) \geq \xi \circ \gamma_{k_2k_1} \circ \xi \circ \gamma_{k_1 i}(\alpha).%
\end{align*}
Proceeding in this manner, we end up with%
\begin{equation*}
  \omega^{\lfloor K_i / B_m \rfloor}(r) \geq \Gamma^{K_i}_{j_i}(s) \geq (\xi \circ \gamma_{j_i k_{l-1}}) \circ (\xi \circ \gamma_{k_{l-1} k_{l-2}}) \circ \cdots \circ (\xi \circ \gamma_{k_1i})(\alpha).%
\end{equation*}
Using Assumption (iii) again, we obtain%
\begin{equation*}
   \min_{0 \leq l < m}(\xi \circ \eta)^l(\alpha) \leq \omega^{\lfloor K_i / B_m \rfloor}(r).%
\end{equation*}
As the right-hand side tends to zero for $K_i \rightarrow \infty$, there exists $K_{\max}$ such that $K_i \leq K_{\max}$ for all $i \in I_{\alpha}$. This implies%
\begin{equation*}
  \Gamma^{K_{\max}+1}_i(s) < \alpha \mbox{\quad for all\ } i \in I_{\alpha}.%
\end{equation*}
Altogether, $\|\Gamma^{K_{\max}+1}(s)\|_{\infty} < \alpha$, which completes the proof.
\end{proof}

\begin{remark}
If the network is symmetric, i.e.~$\gamma_{ij} \neq 0$ if and only if $\gamma_{ji} \neq 0$, we have $\NC^-_i(n) = \NC^+_i(n)$ for all $i,n$. In this case, we almost obtain a characterization of UGAS in terms of a small-gain-like condition from the two preceding propositions.
\end{remark}

Finally, we are interested in the question whether one of the small-gain-like conditions used in the above propositions implies that $\Gamma$ satisfies the $\oplus$-MBI property. We are able to prove the following result.%
 
\begin{proposition}\label{prop_maxmbi}
Let the following assumptions hold:%
\begin{enumerate}
\item[(i)] The system induced by $\hat{\Gamma}$ is UGS.%
\item[(ii)] There exists $n \in \N$ such that for all $s \in \ell_{\infty}^+$ the inequality $s_i \geq \|s\|_{\infty}/2$ implies that $\Gamma_j(s) < s_j$ for some $j \in \NC^+_i(n)$.%
\item[(iii)] The network satisfies standard technical assumptions.%
\end{enumerate}
Then $\Gamma$ satisfies the $\oplus$-MBI property.
\end{proposition}

\begin{proof}
Assume that $s \leq \Gamma(s) \oplus b$ for some $s,b \in \ell_{\infty}^+$. Inductively, this implies%
\begin{equation*}
  \hat{\Gamma}^n(s) \leq \Gamma(\hat{\Gamma}^{n-1}(s)) \oplus b \mbox{\quad for all\ } n \geq 1.%
\end{equation*}
If we look at this inequality componentwise and let $n \rightarrow \infty$, we obtain $\hat{Q}(s) \leq b \oplus \Gamma(\hat{Q}(s))$. Hence, we can without loss of generality assume that $\Gamma(s) \leq s$. Choose $n$ according to Assumption (ii). Then, for every $i \in \N$ with $s_i \geq \|s\|_{\infty}/2$, there exists $j \in \NC_i^+(n)$ with $\Gamma_j(s) < s_j$. This implies $s_j \leq b_j$. Then, there exists $l \in \{0,1,\ldots,n-1\}$ with%
\begin{equation*}
   b_j \geq s_j > \Gamma_j^l(s) \geq (\xi \circ \gamma_{jk_1})\circ(\xi \circ \gamma_{k_1k_2})\circ\cdots\circ(\xi \circ \gamma_{k_{l-1}i})(s_i).%
\end{equation*}
This implies $s_i \leq \max_{0 \leq l < n}(\xi \circ \eta)^{-l}(\|b\|_{\infty}) =: \varphi(\|b\|_{\infty})$. Hence, we obtain $\|s\|_{\infty} = \sup_{s_i \geq \|s\|_{\infty}/2} s_i \leq \varphi(\|b\|_{\infty})$, which completes the proof.
\end{proof}

\begin{remark}
Assumption (ii) in Proposition \ref{prop_maxmbi} is, for instance, satisfied if the system induced by $\Gamma$ is UGES and the network is symmetric. In this case, $\|\Gamma^n(s)\|_{\infty} \leq M \gamma^n \|s\|_{\infty}$ for some $M>0$ and $\gamma \in (0,1)$, and we can choose $n$ such that $M\gamma^n < 1/2$. For symmetric networks, the following statements thus imply each other in the sense that (1) $\Rightarrow$ (2) $\Rightarrow$ (3):%
\begin{enumerate}
\item[(1)] The system induced by $\hat{\Gamma}$ is UGS and the system induced by $\Gamma$ is UGES.%
\item[(2)] $\Gamma$ satisfies the $\oplus$-MBI property.%
\item[(3)] The system induced by $\hat{\Gamma}$ is UGS and the system induced by $\Gamma$ is globally componentwise attractive.
\end{enumerate}
\end{remark}

If the MAFs used to define $\Gamma$ are subadditive, we can show that UGES of the system induced by a slightly enlarged operator implies that the system induced by $\hat{\Gamma}$ is UGS. This is implied by the following proposition in combination with Proposition \ref{prop_smallgain_conds}. The proof is similar to that of \cite[Lem.~13]{dashkovskiy2007iss}.%

\begin{proposition}\label{prop_uges_implies_mbi}
Assume that the MAFs $\mu_i$ are subadditive, i.e.~$\mu_i(s^1 + s^2) \leq \mu_i(s^1) + \mu_i(s^2)$ for all $i \in \N$ and $s^1,s^2 \in \ell_{\infty}^+$. Further, assume that $\# I_i$ is uniformly bounded and that for some $\omega \in \KC_{\infty}$ with $\id - \omega \in \KC_{\infty}$, the system induced by $\Gamma_{\omega}$ is UGES. Then $\Gamma$ satisfies the MBI property.
\end{proposition}

\begin{proof}
To avoid an overload of notation, we only carry out the proof for sum-type operators. It is clear how to generalize the proof to the setup described in the statement of the proposition. By assumption, there are $M > 0$ and $\alpha \in (0,1)$ with%
\begin{equation}\label{eq_uges}
  \|\Gamma_{\omega}^n(s)\|_{\infty} \leq M\alpha^n\|s\|_{\infty} \mbox{\quad for all\ } s \in \ell_{\infty}^+,\ n \geq 0.%
\end{equation}
Fix $m \in \N$ with $M\alpha^m < 1/2$. As before, this implies that for every $s \in \ell_{\infty}^+$ we have the implication%
\begin{equation*}
  s_i \geq \frac{1}{2}\|s\|_{\infty} \quad \Rightarrow \quad \exists j \in \NC^-_i(m) \mbox{ with } \Gamma_j(s) < \omega(s_j).%
\end{equation*}
Now, fix $s,b^0 \in \ell_{\infty}^+$ satisfying $(\id - \Gamma)(s) \leq b^0$. If $\Gamma_j(s) < \omega(s_j)$, this implies%
\begin{equation*}
  (\id - \omega)(s_j) = s_j - \omega(s_j) < s_j - \Gamma_j(s) \leq b_j^0,%
\end{equation*}
and hence $s_j \leq (\id - \omega)^{-1}(b_j^0)$. Here, we use our assumption that $\id - \omega \in \KC_{\infty}$. We put $I^1 := \{ i \in \N : \Gamma_i(s) \geq \omega(s_i) \}$. Then from%
\begin{equation*}
  s_i - \sum_{j \in I_i} \gamma_{ij}(s_j) \leq b_i^0 \mbox{\quad for all\ } i \in \N%
\end{equation*}
it follows that%
\begin{equation*}
  s_i - \sum_{j \in I_i \cap I^1} \gamma_{ij}(s_j) \leq b^0_i + \sum_{j \in I_i \setminus I^1} \gamma_{ij}(s_j) \leq b^0_i + \sum_{j \in I_i \setminus I^1} (\id - \omega)^{-1}(b^0_j) \mbox{\quad for all\ } i \in I^1.%
\end{equation*}
We define $b^1 := (b^0_i + \sum_{j \in I_i \setminus I^1} (\id - \omega)^{-1}(b^0_j))_{i\in I^1} \in \ell_{\infty}^+(I^1)$ and an operator%
\begin{equation*}
  \Gamma_1(s) := \Bigl(\sum_{j \in I_i \cap I^1} \gamma_{ij}(s_j)\Bigr)_{i \in I^1}%
\end{equation*}
which acts on $\ell_{\infty}^+(I^1)$. We claim that the system induced by $\omega^{-1} \circ \Gamma_1$ is UGES with the same constants $M,\alpha$ as in \eqref{eq_uges}. Indeed, from the construction it follows that $\Gamma_1(s|_{I^1}) \leq \Gamma(s)|_{I^1}$ for every $s \in \ell_{\infty}^+$, where $s|_{I^1} = (s_i)_{i\in I^1}$. This implies $\omega^{-1} \circ \Gamma_1(s|_{I^1}) \leq \Gamma_{\omega}(s)|_{I^1}$. Hence, $(\omega^{-1} \circ \Gamma_1)^n \leq \Gamma_{\omega}^n|_{I^1}$ for all $n$, which implies the claim. By our construction,%
\begin{equation*}
  (\id - \Gamma_1)(s|_{I^1}) \leq b^1.%
\end{equation*}
As $s|_{I^1},b^1 \in \ell_{\infty}^+(I^1)$, we can repeat our construction and obtain an index set $I^2 = \{ i \in I^1 : \Gamma_{1,i}(s) \geq \omega(s_i) \}$, a sum-type operator $\Gamma_2:\ell_{\infty}^+(I_2) \rightarrow \ell_{\infty}^+(I_2)$ and a vector $b^2 \in \ell_{\infty}^+(I_2)$ such that $(\id - \Gamma_2)(s|_{I^2}) \leq b^2$. Now, if we start with an index $i \in \N$ such that $s_i \geq (1/2)\|s\|_{\infty}$, at least one $j \in \NC^-_i(m)$ will not be contained in $I^1$. If $i \in I^1$, then again at least one $j \in I^1 \cap \NC^-_i(m)$ will not be contained in $I^2$. Since $\NC^-_i(m)$ is uniformly bounded over $i \in \N$ by our assumption that $\# I_i$ is uniformly bounded, there exists $N \in \N$ such that after at most $N$ steps of this construction, we see that $s_i$ is bounded by $(\id - \omega)^{-1}(b^k_i)$ for some $k \in \{0,1,\ldots,N\}$. By construction, $\|b^{k+1}\|_{\infty} \leq (\id + C (\id - \omega)^{-1})(\|b^k\|_{\infty})$ for each $k$, where $C$ is a bound on $\# I_i$. Hence, there exists $\varphi \in \KC_{\infty}$ with%
\begin{equation*}
  \|s\|_{\infty} = \sup_{s_i \geq \|s\|_{\infty}/2}s_i \leq \varphi(\|b\|_{\infty}),%
\end{equation*}
and this completes the proof.
\end{proof}

\begin{remark}
Since exponential stability is a robust property, it should be sufficient to assume that the system induced by $\Gamma$ is UGES in the above proposition. However, we are not aware of a result that can be directly applied here to guarantee that UGES of $\Gamma$ implies UGES of $\Gamma_{\omega}$ for some small $\omega$.
\end{remark}

Now that we have gained a better understanding of the UGAS property for gain operators, we can approach our main goal, the construction of a path of decay.%

\section{Construction of a path of decay}\label{sec_path}

Throughout this section, we assume that $\Gamma$ satisfies the $\oplus$-MBI property, i.e.~there exists $\varphi \in \KC_{\infty}$ such that%
\begin{equation*}
  s \leq \Gamma(s) \oplus b \quad \Rightarrow \quad \|s\|_{\infty} \leq \varphi(\|b\|_{\infty}) \mbox{\quad for all\ } s,b \in \ell_{\infty}^+.%
\end{equation*}
In particular, by Proposition \ref{prop_hatgamma_ugs}, this implies%
\begin{equation}\label{eq_hatgamma_ugs}
  \|\hat{\Gamma}^n(s)\|_{\infty} \leq \varphi(\|s\|_{\infty}) \mbox{\quad for all\ } s \in \ell_{\infty}^+,\ n \geq 0.%
\end{equation}
We define the first candidate for a path of decay for $\Gamma$ by%
\begin{equation}\label{eq_def_sigma}
  \sigma_*(r) := \hat{Q}(r\unit) = \bigoplus_{k=0}^{\infty}\hat{\Gamma}^k(r\unit),\quad \sigma_*:\R_+ \rightarrow \ell_{\infty}^+.%
\end{equation}
For the analysis of the mapping $\sigma_*$ and, more general, for the construction of paths of decay, it is useful to introduce for each $r > 0$ the operator%
\begin{equation*}
  P_r(s) := r\unit \oplus s,\quad P_r:\ell_{\infty}^+ \rightarrow \ell_{\infty}^+.%
\end{equation*}
The next proposition describes its properties.%

\begin{proposition}\label{prop_pr_props}
The operator $P_r$ is a projection\footnote{The projection onto a closed convex set in $\ell_{\infty}$ is, in general, not unique.} onto the closed convex set $C_r := \{ s \in \ell_{\infty}^+ : s \geq r\unit \}$, i.e.%
\begin{equation*}
  \|s - P_r(s)\|_{\infty} = \inf_{\tilde{s} \in C_r}\|s - \tilde{s}\|_{\infty}.%
\end{equation*}
Moreover, $P_r$ is a continuous monotone operator and for any $s^1,s^2 \in \ell_{\infty}^+$ with $s^1 \leq s^2$, it holds that $P_r(s^2) - P_r(s^1) \leq s^2 - s^1$. Finally, for all $s^1,s^2 \in \ell_{\infty}^+$, it holds that $\|P_r(s^2) - P_r(s^1)\|_{\infty} \leq \|s^2 - s^1\|_{\infty}$.
\end{proposition}

\begin{proof}
It is easy to see that $C_r$ is closed and convex. For any $s \in \ell_{\infty}^+$, we have%
\begin{align*}
  \|P_r(s) - s\|_{\infty} = \|r\unit \oplus s - s\|_{\infty} = \sup_{i \in \N} (\max\{r,s_i\} - s_i) = \sup_{i:\ s_i < r} (r - s_i).%
\end{align*}
For an arbitrary $\tilde{s} \in C_r$, we have%
\begin{align*}
  \|\tilde{s} - s\|_{\infty} &= \sup_{i\in\N} |\tilde{s}_i - s_i| = \max\{ \sup_{i:\ s_i < r} (\tilde{s}_i - s_i), \sup_{i:\ s_i \geq r} |\tilde{s}_i - s_i| \} \\
	&\geq \sup_{i:\ s_i < r} (\tilde{s}_i - s_i) \geq \sup_{i:\ s_i < r} (r - s_i).%
\end{align*}
This completes the proof of the first statement. The continuity and monotonicity of $P_r$ are trivial. Now, let $s^1,s^2 \in \ell_{\infty}^+$ with $s^1 \leq s^2$. Then%
\begin{equation*}
  \pi_i(P_r(s^2) - P_r(s^1)) = \max\{r,s^2_i\} - \max\{r,s^1_i\} \leq \max\{r-r,s^2_i-s^1_i\} = s^2_i - s^1_i%
\end{equation*}
for all $i \in \N$, which implies $P_r(s^2) - P_r(s^1) \leq s^2 - s^1$. Finally, consider arbitrary $s^1,s^2 \in \ell_{\infty}^+$. Then%
\begin{equation*}
  \max\{r,s^1_i\} - \max\{r,s^2_i\} \leq \max\{r-r,s^1_i - s^2_i\} \leq |s^1_i - s^2_i| \leq \|s^1 - s^2\|_{\infty}%
\end{equation*}
for all $i \in \N$. The same estimate holds with $s^1_i$ and $s^2_i$ interchanged. Hence, $\|P_r(s^1) - P_r(s^2)\|_{\infty} \leq \|s^1 - s^2\|_{\infty}$.
\end{proof}

We also define the operator $\Gamma_r := P_r \circ \Gamma:\ell_{\infty}^+ \rightarrow \ell_{\infty}^+$ for each $r > 0$, which is again a continuous monotone operator.%

By construction, $\sigma_*$ already satisfies some of the properties of a path of decay, as summarized in the following proposition.%

\begin{proposition}\label{prop_psd_elprops}
The mapping $\sigma_*$ has the following properties:%
\begin{enumerate}
\item[(i)] $\Gamma(\sigma_*(r)) \leq \sigma_*(r)$ for all $r \geq 0$.%
\item[(ii)] For all $i \in \N$ and $r \geq 0$, we have $r \leq \sigma_{*,i}(r) \leq \varphi(r)$.
\item[(iii)] Each $\sigma_{*,i}$ satisfies $\sigma_{*,i}(0) = 0$, $\sigma_{*,i}(r) > 0$ for all $r>0$, and $\sigma_{*,i}(r) \rightarrow \infty$ as $r \rightarrow \infty$.%
\item[(iv)] Each $\sigma_{*,i}$ is a lower semicontinuous and monotonically non-decreasing function.%
\item[(v)] $\sigma_*$ is continuous at $r=0$.%
\item[(vi)] The point $\sigma_*(r)$ is a fixed point of $\Gamma_r$ for each $r \geq 0$.
\end{enumerate}
\end{proposition}

\begin{proof}
(i) This follows immediately from the construction (as $\hat{Q}(r\unit)$ is a fixed point of $\hat{\Gamma}$).%

(ii) The statement holds, because $r\unit \leq \hat{Q}(r\unit) = \bigoplus_{k=0}^{\infty}\hat{\Gamma}^k(r\unit) \leq \varphi(r)\unit$.%

(iii) This directly follows from (ii).%

(iv) We can write $\sigma_{*,i}(r) = \sup_{k\geq0}\hat{\Gamma}^k_i(r\unit)$. As a supremum over continuous functions, $\sigma_{*,i}$ is lower semicontinuous. If $r_1 \leq r_2$, then%
\begin{equation*}
  \sigma_{*,i}(r_1) = \sup_{k\geq0}\hat{\Gamma}^k_i(r_1\unit) \leq \sup_{k\geq0}\hat{\Gamma}^k_i(r_2\unit) = \sigma_{*,i}(r_2),%
\end{equation*}
where we use the fact that $\hat{\Gamma}^k$ is a monotone operator for each $k \geq 0$.

(v) This follows from%
\begin{align*}
  \|\sigma_*(r) - \sigma_*(0)\|_{\infty} = \|\sigma_*(r)\|_{\infty} = \sup_{i\in\N}\sigma_{*,i}(r) \leq \varphi(r).
\end{align*}

(vi) By Proposition \ref{prop_elem}(iv), we have $\hat{\Gamma}^n(r\unit) = r\unit \oplus \Gamma(\hat{\Gamma}^{n-1}(r\unit))$ for each $n \in \N$. By Lemma \ref{lem_weak_continuity}, the componentwise convergence $\hat{\Gamma}^n(r\unit) \rightarrow \sigma_*(r)$ implies that $\Gamma(\hat{\Gamma}^{n-1}(r\unit))$ converges componentwise to $\Gamma(\sigma_*(r))$. Hence,%
\begin{equation*}
  \sigma_{*,i}(r) = \lim_{n \rightarrow \infty} \hat{\Gamma}^n_i(r\unit) = \max\{r,\lim_{n \rightarrow \infty} \Gamma_i(\hat{\Gamma}^{n-1}(r\unit)) \} = \max\{r,\Gamma_i(\sigma_*(r)) \}%
\end{equation*}
for each $i \in \N$. This is equivalent to $\Gamma_r(\sigma_*(r)) = \sigma_*(r)$.
\end{proof}

We now introduce the second candidate for a path of decay by%
\begin{equation*}
  \sigma^*_i(r) := \inf_{n \in \Z_+}\Gamma_{r,i}^n(\varphi(r)\unit)%
\end{equation*}
for all $i \in \N$ and $r \geq 0$. Although the definition of $\sigma^*$ looks different than that of $\sigma_*$, it is quite similar, since we can write $\sigma_{*,i}(r) = \sup_{n \in \Z_+}\Gamma_{r,i}^n(r\unit)$ for each $i \in \N$. The next lemma is crucial.%

\begin{lemma}
The point $\sigma^*(r)$ is a fixed point of $\Gamma_r$ for each $r \geq 0$.
\end{lemma}

\begin{proof}
For $r = 0$, this is trivial. Now, fix $r > 0$ and consider the sequence $s^n := \Gamma_r^n(\sigma_*(\varphi(r)))$, $n \in \Z_+$. We have%
\begin{equation*}
  \Gamma_r(\sigma_*(\varphi(r))) = r\unit \oplus \Gamma(\sigma_*(\varphi(r))) \leq r\unit \oplus \sigma_*(\varphi(r)) = \sigma_*(\varphi(r)),%
\end{equation*}
where the last equality follows from $\sigma_*(\varphi(r)) \geq \varphi(r)\unit \geq r\unit$. By monotonicity of $\Gamma_r$, it follows that the sequence $(s^n)_{n\in\Z_+}$ is componentwise non-increasing. Hence, it converges componentwise to a fixed point $s^*$ of $\Gamma_r$ (using similar arguments as in the proof of Lemma \ref{lem_weak_continuity}). We claim that $s^* = \sigma^*(r)$. Clearly, we have $\Gamma_r^n(\varphi(r)\unit) \leq s^n$ for all $n \in \Z_+$. If $s$ is any fixed point of $\Gamma_r$, then the $\oplus$-MBI property implies $s \leq \varphi(r)\unit$. Hence, $s^* \leq \varphi(r)\unit$, implying $s^* = \Gamma_r^n(s^*) \leq \Gamma_r^n(\varphi(r)\unit) \leq s^n$ for all $n \in \Z_+$. It follows that the componentwise infimum of $\Gamma_r^n(\varphi(r)\unit)$ equals $s^*$, i.e.~$\sigma^*(r) = s^*$ as claimed. 
\end{proof}

The fundamental properties of the mapping $\sigma^*$ can be proved analogously as those of $\sigma_*$ and are summarized in the next proposition.%

\begin{proposition}
The mapping $\sigma^*$ has the following properties:%
\begin{enumerate}
\item[(i)] $\Gamma(\sigma^*(r)) \leq \sigma^*(r)$ for all $r \geq 0$.%
\item[(ii)] For all $i \in \N$ and $r \geq 0$, we have $r \leq \sigma^*_i(r) \leq \varphi(r)$.
\item[(iii)] Each $\sigma^*_i$ satisfies $\sigma^*_i(0) = 0$, $\sigma^*_i(r) > 0$ for all $r>0$, and $\sigma^*_i(r) \rightarrow \infty$ as $r \rightarrow \infty$.%
\item[(iv)] Each $\sigma^*_i$ is an upper semicontinuous and monotonically non-decreasing function.%
\item[(v)] $\sigma^*$ is continuous at $r=0$.%
\end{enumerate}
\end{proposition}

The following proposition yields some information about the fixed point set $\Fix(\Gamma_r)$ of the operator $\Gamma_r$.%

\begin{proposition}
For each $r > 0$, the points $\sigma_*(r)$ and $\sigma^*(r)$ are the minimal and the maximal fixed point of $\Gamma_r$, respectively. That is, any fixed point $s$ satisfies $\sigma_*(r) \leq s \leq \sigma^*(r)$. In particular, the order interval $[\sigma_*(r),\sigma^*(r)] := \{ s \in \ell_{\infty}^+ : \sigma_*(r) \leq s \leq \sigma^*(r) \}$ is an invariant set for $\Gamma_r$ and contains all of its fixed points. 
\end{proposition}

\begin{proof}
Let $s$ be an arbitrary fixed point of $\Gamma_r$. Then $s = \Gamma_r(s) \geq r\unit$, implying $s = \Gamma_r^n(s) \geq \Gamma_r^n(r\unit)$ for all $n \in \Z_+$. In the limit, we obtain that $s \geq \sigma_*(r)$. Moreover, since any fixed point $s$ must satisfy $s \leq \varphi(r)\unit$, we have $s = \Gamma_r^n(s) \leq \Gamma_r^n(\varphi(r)\unit)$ for all $n \in \Z_+$, implying that $\sigma^*(r)$ is the maximal fixed point of $\Gamma_r$. It is now easy to see that $[\sigma_*(r),\sigma^*(r)]$ is invariant under $\Gamma$.
\end{proof}

We now present our main result about the existence of a path of decay for $\Gamma$.%

\begin{theorem}\label{thm_path_of_decay}
Let $\Gamma$ be a well-defined and continuous gain operator satisfying the following assumptions:%
\begin{enumerate}
\item[(i)] The family $\{\gamma_{ij} : i,j \in \N\}$ is pointwise equicontinuous.%
\item[(ii)] $\Gamma$ satisfies the $\oplus$-MBI property.%
\item[(iii)] There exists $\omega \in \KC_{\infty}$ with $\omega < \id$ such that the system induced by $\Gamma_{\omega} = \omega^{-1} \circ \Gamma$ is UGAS.%
\item[(iv)] For all $0 < R_1 < R_2$, there exists $l>0$ such that the MAFs $\mu_i$ satisfy%
\begin{equation*}
  (\mu_i(s^2) - \mu_i(s^1)) \geq l(s^2_j - s^1_j),%
\end{equation*}
whenever $j \in I_i$ and $R_1\unit \leq s^1 \leq s^2 \leq R_2\unit$.%
\item[(v)] For every compact interval $J \subset (0,\infty)$, there exists $c > 0$ such that%
\begin{equation*}
  |\gamma_{ij}(r_1) - \gamma_{ij}(r_2)| \geq c|r_1 - r_2|%
\end{equation*}
for all $i \in \N$, $j\in I_i$, and $r_1,r_2 \in J$.%
\end{enumerate}
Then every single-valued selection of the set-valued mapping $r \mapsto \Fix(\Gamma_r)$ is a path of decay for $\Gamma$ provided that it is locally Lipschitz continuous on $(0,\infty)$.
\end{theorem}

\begin{proof}
We verify the four properties of a path of decay:%

Since $\sigma(r) = \Gamma_r(\sigma(r))$ for each $r \geq 0$, we have $\Gamma(\sigma(r)) \leq r\unit \oplus \Gamma(\sigma(r)) = \sigma(r)$ so that property (P1) of a path of decay is satisfied.%

As every fixed point $s$ of $\Gamma_r$ satisfies $r\unit \leq s \leq \varphi(r)\unit$, it follows that $\sigma$ satisfies property (P2) of a 
path of decay with $\sigma_{\min} = \id$ and $\sigma_{\max} = \varphi$.%

By assumption, $\sigma$ is locally Lipschitz continuous on $(0,\infty)$. Hence, for every compact interval $J \subset (0,\infty)$, there exists $C > 0$ with%
\begin{equation*}
  |\sigma_i(r_1) - \sigma_i(r_2)| \leq \|\sigma(r_1) - \sigma(r_2)\|_{\infty} \leq C |r_1 - r_2|%
\end{equation*}
for all $r_1,r_2 \in J$ and $i \in \N$. This shows half of property (P4) (see Remark \ref{rem_pod2}). Moreover, it shows that each $\sigma_i$ is continuous. Since $\sigma_i(r) \geq r$, it also follows that $\sigma_i(r) \rightarrow \infty$ as $r \rightarrow \infty$. If we can show the second half of property (P4), i.e.~the lower Lipschitz estimate for $\sigma_i$, both (P3) and (P4) will be fully proved.%

We first prove the following claim: For every $r_0 > 0$, there are $m \in \N$ and $\rho \in (0,r_0)$ such that for every $i \in \N$ there is $j \in \NC^-_i(m)$ with $\sigma_j(r) = r$ for all $r \in [r_0 - \rho,r_0 + \rho]$.%

Let $m = m(r_0,\varphi(r_0))$ be chosen such that the inequalities $s_i \geq r_0$ and $\|s\|_{\infty} \leq \varphi(r_0)$ imply that for every $i \in \N$ there is $j \in \NC^-_i(m)$ with $\Gamma_j(s) < \omega(s_j)$, where we use Proposition \ref{prop_ugas_char_part1} and Assumption (iii). Then, for every $i \in \N$ there exists $j_i \in \NC^-_i(m)$ with $\Gamma_{j_i}(\sigma(r_0)) < \omega(\sigma_{j_i}(r_0))$. To prove the claim, it suffices to show that there exists $\rho > 0$ with $\Gamma_{j_i}(\sigma(r)) < \sigma_{j_i}(r)$ for all $r \in [r_0 - \rho,r_0 + \rho]$ and $i \in \N$ because of the relation $\sigma_{j_i}(r) = \max\{r,\Gamma_{j_i}(\sigma(r))\}$. First, we can find $\ep > 0$ such that%
\begin{equation}\label{eq_ineq1}
  \sigma_{j_i}(r_0) - \Gamma_{j_i}(\sigma(r_0)) \geq \ep \mbox{\quad for all\ } i \in \N.%
\end{equation}
This can be achieved if we choose $\ep>0$ such that $\omega(r) \leq r - \ep$ for all $r \in [r_0,\varphi(r_0)]$. From the continuity of $\sigma$ and $\Gamma$, it follows that there exists $\rho > 0$ with%
\begin{equation}\label{eq_ineq2}
  \|\sigma(r) - \sigma(r_0)\|_{\infty} < \frac{\ep}{3} \mbox{\quad for all\ } r \in [r_0 - \rho,r_0 + \rho]%
\end{equation}
and%
\begin{equation}\label{eq_ineq3}
  \|\Gamma(\sigma(r)) - \Gamma(\sigma(r_0))\|_{\infty} < \frac{\ep}{3} \mbox{\quad for all\ } r \in [r_0 - \rho,r_0 + \rho].%
\end{equation}
Hence, for any $i \in \N$ and $r \in [r_0 - \rho,r_0 + \rho]$, the inequalities \eqref{eq_ineq1}, \eqref{eq_ineq2} and \eqref{eq_ineq3} together imply%
\begin{align*}
  \sigma_{j_i}(r) - \Gamma_{j_i}(\sigma(r)) \geq \sigma_{j_i}(r_0) - \frac{\ep}{3} - \Gamma_{j_i}(\sigma(r_0)) - \frac{\ep}{3} \geq \frac{\ep}{3} > 0.
\end{align*}
This proves the claim.%

Now, assume that $j = j_i \in I_i$ for some $i \in \N$. As $\sigma(r)$ is a fixed point of $\Gamma_r$, we have%
\begin{equation*}
  \sigma_i(r) = \max\{r,\Gamma_i(\sigma(r))\}.%
\end{equation*}
If the function $\alpha(r) := \Gamma_i(\sigma(r))$ on some interval $J$ satisfies $|\alpha(r_2) - \alpha(r_1)| \geq a|r_2 - r_1|$ for a constant $a>0$, then $|\sigma_i(r_2) - \sigma_i(r_1)| \geq \min\{1,a\}|r_2 - r_1|$ on $J$. Now, if $r_1 < r_2$ with $r_1,r_2 \in [r_0 - \rho,r_0 + \rho]$, then%
\begin{align*}
   \Gamma_i(\sigma(r_2)) - \Gamma_i(\sigma(r_1)) &= \mu_i( [\gamma_{ik}(\sigma_k(r_2))]_{k\in\N} ) - \mu_i( [\gamma_{ik}(\sigma_k(r_1))]_{k\in\N} ) \\
																&\geq l(\gamma_{ij}(r_2) - \gamma_{ij}(r_1)) \geq lc (r_2 - r_1),%
\end{align*}
where $l$ and $c$ come from Assumption (iv) and (v), respectively. To use Assumption (iv), we note the following:%
\begin{itemize}
\item For all $k \in I_i$, we have $\gamma_{ik}(\sigma_k(r_0+\rho)) \leq \gamma_{ik}(\varphi(r_0 + \rho))$ and the latter is uniformly bounded from above, which follows from Assumption (i).%
\item For all $k \in I_i$, we have%
\begin{align*}
  \gamma_{ik}(\sigma_k(r_0-\rho)) \geq \gamma_{ik}(r_0 - \rho) \geq \gamma_{ik}(r_0 - \rho) - \gamma_{ik}(\frac{1}{2}(r_0-\rho))
	\geq \frac{\hat{c}}{2}(r_0 - \rho),%
\end{align*}
where $\hat{c}$ comes from Assumption (v).%
\end{itemize}
We conclude that%
\begin{equation*}
  |\sigma_i(r_1) - \sigma_i(r_2)| \geq \min\{1,lc\} |r_1 - r_2| \mbox{\quad for all\ } r_1,r_2 \in [r_0 - \rho,r_0 + \rho].%
\end{equation*}
Now, if $j \in I_k$ and $k \in I_i$, then%
\begin{align*}
  \Gamma_i(\sigma(r_2)) - \Gamma_i(\sigma(r_1)) &\geq l(\gamma_{ik}(\sigma_k(r_2)) - \gamma_{ik}(\sigma_k(r_1))) \\
	&\geq lc (\sigma_k(r_2) - \sigma_k(r_1)) \geq \min\{lc,(lc)^2\}(r_2 - r_1),%
\end{align*}
implying%
\begin{equation*}
  |\sigma_i(r_1) - \sigma_i(r_2)| \geq \min\{1,lc,(lc)^2\} |r_1 - r_2|.%
\end{equation*}
Without loss of generality, we can assume that $lc < 1$. Then, inductively we obtain%
\begin{equation*}
  |\sigma_i(r_1) - \sigma_i(r_2)| \geq (lc)^m |r_1 - r_2|%
\end{equation*}
for all $i \in \N$ and $r_1,r_2 \in [r_0 - \rho,r_0 + \rho]$, which completes the proof.
\end{proof}

\begin{remark}
Assumption (iv) of the theorem is, for instance, satisfied for sum-type operators. It is not necessarily satisfied for max-type operators.
\end{remark}

The simplest case in Theorem \ref{thm_path_of_decay} is the one when there is a unique fixed point for each $\Gamma_r$, or equivalently, when $\sigma_*(r) = \sigma^*(r)$. The following corollary discusses this case for subadditive and homogeneous gain operators.%

\begin{corollary}
Assume that $\Gamma$ is a subadditive and homogeneous gain operator, which is well-defined and continuous, and satisfies the following assumptions:%
\begin{enumerate}
\item[(i)] $\inf_{n \in \N} \|\Gamma^n(\unit)\|_{\infty} < 1$.%
\item[(ii)] There exists $c > 1$ with $c^{-1} \leq \gamma_{ij} \leq c$ for all $i\in\N$ and $j \in I_i$.%
\item[(iii)] For all $0 < R_1 < R_2$, there exists $l>0$ such that the MAFs $\mu_i$ satisfy%
\begin{equation*}
  (\mu_i(s^2) - \mu_i(s^1)) \geq l(s^2_j - s^1_j),%
\end{equation*}
whenever $j \in I_i$ and $R_1\unit \leq s^1 \leq s^2 \leq R_2\unit$.
\end{enumerate}
Then $\sigma := \sigma_* = \sigma^*$ is a globally Lipschitz continuous path of decay for $\Gamma$. 
\end{corollary}

\begin{proof}
Assumption (i) implies that the system induced by $\Gamma$ is UGES, see \cite[Prop.~9]{mironchenko2021iss}. If $s \leq \Gamma(s) \oplus b$, then $s \leq \Gamma(s) + b$. Using the subadditivity of $\Gamma$, we obtain inductively that $s \leq \Gamma^n(s) + \sum_{k=0}^{n-1} \Gamma^k(b)$ for all $n \in \Z_+$. Since $\Gamma^n(s) \rightarrow 0$ for $n \rightarrow \infty$, this implies $s \leq \sum_{k=0}^{\infty}\Gamma^k(b)$. Since the system induced by $\Gamma$ is UGES, we obtain%
\begin{equation*}
  \|s\|_{\infty} \leq \frac{M}{1-\gamma}\|b\|_{\infty}%
\end{equation*}
for some $M > 0$ and $\gamma \in (0,1)$. In particular, $\Gamma$ satisfies the $\oplus$-MBI property. Now, assume that $s^1 \leq s^2$ are two fixed points of $\Gamma_r$ for some $r > 0$. Then%
\begin{equation*}
  s^2 - s^1 = \Gamma_r(s^2) - \Gamma_r(s^1) \leq \Gamma(s^2) - \Gamma(s^1) \leq \Gamma(s^2 - s^1),%
\end{equation*}
where we used Proposition \ref{prop_pr_props} and $\Gamma(s^2) = \Gamma(s^1 + (s^2 - s^1)) \leq \Gamma(s^1) + \Gamma(s^2 - s^1)$. Since $\Gamma$ satisfies the small-gain condition (as a special case of the $\oplus$-MBI property), it follows that $s^1 = s^2$. Hence, $\sigma(r) := \sigma_*(r) = \sigma^*(r)$ is well-defined. For any $r_1 < r_2$ in $\R_+$, we obtain%
\begin{align*}
  \sigma(r_2) - \sigma(r_1) &= \Gamma_{r_2}(\sigma(r_2)) - \Gamma_{r_1}(\sigma(r_1)) \leq (r_2 - r_1)\unit \oplus (\Gamma(\sigma(r_2)) - \Gamma(\sigma(r_1))) \\
	&\leq (r_2 - r_1)\unit \oplus \Gamma(\sigma(r_2) - \sigma(r_1)) \\
	&\leq (r_2 - r_1)\unit \oplus \Gamma( (r_2 - r_1)\unit \oplus \Gamma(\sigma(r_2) - \sigma(r_1)) ) \\
	&\leq (r_2 - r_1)\unit \oplus \Gamma( (r_2 - r_1)\unit + \Gamma(\sigma(r_2) - \sigma(r_1)) \\
	&\leq (r_2 - r_1)\unit + \Gamma( (r_2 - r_1)\unit ) + \Gamma^2(\sigma(r_2) - \sigma(r_1)).%
\end{align*}
Inductively, we obtain%
\begin{equation*}
  \sigma(r_2) - \sigma(r_1) \leq \sum_{k=0}^n \Gamma^k((r_2 - r_1)\unit) + \Gamma^{n+1}(\sigma(r_2) - \sigma(r_1))%
\end{equation*}
for all $n \geq 0$. Hence, in the limit for $n \rightarrow \infty$, we have%
\begin{align*}
  \|\sigma(r_2) - \sigma(r_1)\|_{\infty} &\leq \sum_{k=0}^{\infty} \|\Gamma^k ((r_2 - r_1)\unit)\|_{\infty} \\
	&\leq \sum_{k=0}^{\infty} M \gamma^k |r_2 - r_1| = \frac{M}{1 - \gamma}|r_2 - r_1|.%
\end{align*}
It follows that $\sigma$ is globally Lipschitz continuous. It remains to show that Assumptions (i) and (iii) of Theorem \ref{thm_path_of_decay} are satisfied:%
\begin{itemize}
\item The equicontinuity of $\{\gamma_{ij}\}$ follows from Assumption (ii) of this theorem.%
\item Choose a linear $\omega \in (\gamma,1)$. Since $\Gamma$ is assumed to be homogeneous, $\Gamma_{\omega}^n(s) = \omega^{-n} \Gamma^n(s)$ for all $s$ and $n$. This easily implies that also the system induced by $\Gamma_{\omega}$ is UGES and thus UGAS.%
\end{itemize}
The proof is complete.
\end{proof}

\begin{remark}
It should be mentioned that subadditive and homogeneous gain operators also admit linear paths of decay under the given assumptions, which is shown in \cite{mironchenko2021iss}. The path $\sigma$ in the above result, in general, cannot expected to be linear.
\end{remark}

Specializing the preceding result to linear gain operators, we obtain the following corollary.%

\begin{corollary}
Assume that $\Gamma$ is a linear and continuous gain operator satisfying the following assumptions:%
\begin{enumerate}
\item[(i)] The spectral radius of $\Gamma$ satisfies $r(\Gamma) < 1$.
\item[(ii)] There exists $c > 1$ with $c^{-1} \leq \gamma_{ij} \leq c$ for all $i\in\N$ and $j \in I_i$.%
\end{enumerate}
Then $\sigma := \sigma_* = \sigma^*$ is a globally Lipschitz continuous path of decay for $\Gamma$. 
\end{corollary}

\begin{remark}
Actually, the estimate $\gamma_{ij} \leq c$ in Assumption (ii) of the above corollary is redundant, since this follows from the required continuity of $\Gamma$.
\end{remark}

Although Theorem \ref{thm_path_of_decay} is not applicable to max-type operators, because the corresponding MAFs do not satisfy Assumption (iv), we can show that $\sigma_*(r) = \sigma^*(r)$ is a path of decay for a max-type operator under mild assumptions.%

\begin{theorem}
Assume that $\Gamma$ is a well-defined and continuous max-type gain operator satisfying the following assumptions:%
\begin{enumerate}
\item[(i)] The system induced by $\Gamma$ is UGAS.%
\item[(ii)] For each compact interval $J \subset (0,\infty)$, there are $0 < l < L$ such that%
\begin{equation*}
  l|r_1 - r_2| \leq |\gamma_{ij}(r_1) - \gamma_{ij}(r_2)| \leq L|r_1 - r_2|%
\end{equation*}
for all $r_1,r_2 \in J$ and $i \in \N$, $j \in I_i$.
\end{enumerate}
Then $\sigma(r) := \sigma_*(r) = \sigma^*(r)$ is a path of decay for $\Gamma$.
\end{theorem}

\begin{proof}
First, observe that Assumption (i) implies that $\Gamma$ satisfies the $\oplus$-MBI property, see Proposition \ref{prop_hatgamma_ugs}. Now, assume that $s \in \ell_{\infty}^+$ is a fixed point of $\Gamma_r$ for some $r > 0$. Then, using that $\Gamma$ is a max-preserving operator, one can show that%
\begin{equation*}
  s = \Gamma_r^n(s) = \bigoplus_{k=0}^{n-1}\Gamma^k(r\unit) \oplus \Gamma^n(s) \mbox{\quad for all\ } n \in \Z_+.%
\end{equation*}
Since $\Gamma^n(s)$ converges to zero as $n \rightarrow \infty$ by Assumption (i), we obtain%
\begin{equation*}
  s = \bigoplus_{k=0}^{\infty}\Gamma^k(r\unit) = \sigma_*(r) = \sigma^*(r) = \sigma(r).%
\end{equation*}
Now, \cite[Thm.~VI.1]{kawan2021lyapunov} shows that $\sigma$ is a path of decay. The main idea used there to show that $\sigma$ satisfies (P3) and (P4) is that for every $r > 0$ there exists $N(r) \in \N$ with%
\begin{equation*}
  \sigma(r) = \bigoplus_{k=0}^{N(r)}\Gamma^k(r\unit),%
\end{equation*}
which follows immediately from Assumption (i). Moreover, $N(r)$ can be chosen independently of $r$ on a compact interval $J \subset (0,\infty)$. This implies that each $\sigma_i$ can be written as the maximum of only finitely many strictly increasing and continuous functions, which yields $\sigma_i \in \KC_{\infty}$. From Assumption (ii), we can conclude that these finitely many functions satisfy uniform Lipschitz estimates locally, so the same must be true for the functions $\sigma_i$.
\end{proof}

We can also specialize Theorem \ref{thm_path_of_decay} to finite networks for which the developed theory also applies (as we can extend them in a trivial way to infinite networks). One way to specialize the result to finite networks is as follows.%

\begin{corollary}
Consider a finite network with associated gain operator $\Gamma:\R^N_+ \rightarrow \R^N_+$ and let the following assumptions hold:%
\begin{enumerate}
\item[(i)] The MAFs $\mu_i$, $i = 1,\ldots,N$, are subadditive and satisfy $\mu_i(s) \rightarrow \infty$ if $\|s\|_{\infty} \rightarrow \infty$.%
\item[(ii)] There exists $\rho \in \KC_{\infty}$ such that $(\id + \rho) \circ \Gamma$ satisfies the SGC.%
\item[(iii)] For all $0 < R_1 < R_2$, there exists $l>0$ such that the MAFs $\mu_i$ satisfy%
\begin{equation*}
  (\mu_i(s^2) - \mu_i(s^1)) \geq l(s^2_j - s^1_j),%
\end{equation*}
whenever $j \in I_i$ and $R_1\unit \leq s^1 \leq s^2 \leq R_2\unit$.%
\item[(iv)] The functions $\gamma_{ij}$ are $C^1$ on $(0,\infty)$ with $\gamma_{ij}'(r) > 0$ for all $r > 0$.%
\end{enumerate}
Then every single-valued selection of the set-valued mapping $r \mapsto \Fix(\Gamma_r)$ is a path of decay for $\Gamma$ provided that it is locally Lipschitz continuous on $(0,\infty)$.
\end{corollary}

\begin{proof}
As the family $\{\gamma_{ij}\}$ only consists of finitely many functions, it is automatically equicontinuous. Then also the assumption of Proposition \ref{prop_gamma_welldefined} is satisfied, showing that $\Gamma$ is well-defined. As $\Gamma$ is composed of finitely many continuous functions, it is continuous.%

From Assumptions (i) and (ii), it follows by \cite[Thm.~6.1]{ruffer2010monotone} that $\Gamma$ satisfies the MBI-property, which in turn implies the $\oplus$-MBI property.%

By \cite[Thm.~4.6]{ruffer2010monotone}, Assumption (ii) also implies that the system induced by $(\id + \tilde{\rho}) \circ \Gamma$ is UGAS for some $\tilde{\rho} \in \KC_{\infty}$. Finally, Assumption (iv) implies that the functions $\gamma_{ij}$ satisfy local Lipschitz estimates from below.
\end{proof}

For finite networks, the assumption that $\Gamma_r$ has only one fixed point for each $r>0$ already implies that $\sigma$ is continuous, which is an easy consequence of the following proposition.%

\begin{proposition}
Assume that for each $r > 0$, $\Gamma_r$ has only one fixed point $\sigma(r) = \sigma_*(r) = \sigma^*(r)$ and the following two sequences of functions converge to $\sigma(r)$ with respect to the $\ell_{\infty}$-norm: (1) $f_n(r) = \Gamma_r^n(r\unit)$ and (2) $g_n(r) = \Gamma_r^n(\varphi(r)\unit)$. Then $\sigma$ is continuous.
\end{proposition}

\begin{proof}
We fix $r_0 > 0$, $\ep > 0$ and two vectors $s^1 \ll \sigma(r_0) \ll s^2$ such that $\|s^1 - s^2\|_{\infty} \leq \ep$, where the notation $s \ll t$ means that $t - s \in \inner(\ell_{\infty}^+)$. By assumption, there exists $n \in \Z_+$ with $s^1 \ll f_n(r_0) \leq g_n(r_0) \ll s^2$. Since $f_n$ and $g_n$ are continuous, there exists a neighborhood $J \subset (0,\infty)$ of $r_0$ with $s^1 \ll f_n(r) \leq g_n(r) \ll s^2$ for all $r \in J$. It follows that for all $r \in J$ with $r > r_0$ we have $\sigma(r) - \sigma(r_0) \leq g_n(r) - f_n(r_0) \ll s^2 - s^1$, which implies $\|\sigma(r) - \sigma(r_0)\|_{\infty} \leq \|s^2 - s^1\| \leq \ep$. Analogously, if $r \in J$ and $r < r_0$, then $\sigma(r_0) - \sigma(r) \leq g_n(r_0) - f_n(r) \ll s^2 - s^1$. This concludes the proof.
\end{proof}

Finally, we take a look at sum-type operators. It seems to us that the requirement that $\Gamma_r$ has only one fixed point for each $r > 0$ is unnecessarily strong in this case (although we are not able to prove it). We first investigate what the assumption that $\Gamma$ induces a UGAS system implies for the fixed point set of $\Gamma_r$ (in the general case).%

\begin{proposition}
Assume that $\Gamma$ satisfies the $\oplus$-MBI property and the system induced by $\Gamma$ is UGAS. Then, there exists $n = n(r)$ such that for every $i \in \N$ there is $j \in \NC^-_i(n)$ such that for any fixed point $s \in \Fix(\Gamma_r)$ we have $s_j = \sigma_{*,j}(r) = r$.
\end{proposition}

\begin{proof}
Consider the maximal fixed point $s^* := \sigma^*(r)$, which satisfies $r\unit \leq s \leq \varphi(r)\unit$. By Proposition \ref{prop_ugas_char_part1}, UGAS implies the existence of $n = n(r,\varphi(r)) \in \N$ such that for every $i \in \N$ there is $j \in \NC^-_i(n)$ with $\Gamma_j(s^*) < s^*_j$. This implies $s^*_j = r$, and hence $r \leq \sigma_{*,j}(r) \leq s^*_j = r$. Since any other fixed point lies between $\sigma_*(r)$ and $\sigma^*(r)$, all fixed points agree with $\sigma_*(r)$ in the component $j$.  
\end{proof}

Now, we look at sum-type operators. Let us take a closer look at the fixed point equation $s = \Gamma_r(s)$ and try to characterize its unique solvability. In components, the fixed point equation reads%
\begin{equation*}
  s_i = \max\{r,\Gamma_i(s)\},\quad i \in \N.%
\end{equation*}
Let us now define the index set $\IC = \IC(s) := \{ i \in \N : \Gamma_i(s) > r \}$. We then have%
\begin{equation*}
  s_i = \left\{\begin{array}{cc}
	               \Gamma_i(s) & \mbox{if } i \in \IC, \\
									r & \mbox{if } i \notin \IC.
								\end{array}\right.%
\end{equation*}
Hence, the possible non-uniqueness of $s$ is only an issue on the index set $\IC$. Here,%
\begin{equation*}
  s_i = \Gamma_i(s) = \sum_{j \in I_i}\gamma_{ij}(s_j) = \sum_{j \in I_i \cap \IC} \gamma_{ij}(s_j) + \sum_{j \in I_i \setminus \IC} \gamma_{ij}(r).%
\end{equation*}
The equation can thus be rewritten as%
\begin{equation*}
  s_i - \sum_{j \in I_i \cap \IC}\gamma_{ij}(s_j) = \sum_{j \in I_i \setminus \IC}\gamma_{ij}(r).%
\end{equation*}
We define an operator $\Gamma_{\IC}$ on $\ell_{\infty}^+(\IC)$ by%
\begin{equation*}
  \Gamma_{\IC}(s) := \Bigl(\sum_{j \in I_i \cap \IC} \gamma_{ij}(s_j)\Bigr)_{i \in \IC}.%
\end{equation*}
Then the above implies%
\begin{equation*}
  (\id - \Gamma_{\IC})(s|_{\IC}) = h_{\IC}(r) := \Bigl(\sum_{j \in I_i \setminus \IC}\gamma_{ij}(r)\Bigr)_{i \in \IC}.%
\end{equation*}
This is equivalent to $s|_{\IC} \in (\id - \Gamma_{\IC})^{-1}(h_{\IC}(r))$. Since the index set $\IC$ depends on $s$, the full fixed point set of $\Gamma_r$ can be written as%
\begin{equation*}
  \Fix(\Gamma_r) = \bigcup_{\IC \subseteq \N}\Bigl( (\id - \Gamma_{\IC})^{-1}(h_{\IC}(r)) \cap \{s : \Gamma_i(s) > r \Leftrightarrow i \in \IC \} \Bigr),%
\end{equation*}
where the union is a disjoint union. It seems to be a too strong requirement that this union contains only one point.%

Nevertheless, we end the paper with a sufficient condition under which the fixed points of $\Gamma_r$ are unique for a sum-type operator and a path of decay exists.%

\begin{theorem}
Assume that $\Gamma$ is a sum-type gain operator satisfying the $\oplus$-MBI property. Additionally to assumptions (i), (iii) and (v) in Theorem \ref{thm_path_of_decay}, let the following assumptions hold:%
\begin{enumerate}
\item[(i)] $\Gamma^2 = \Gamma \circ \Gamma$ is an order contraction on each of the order intervals $[R_1\unit,R_2\unit]$. That is, there exist constants $K = K(R_1,R_2) \in (0,1)$ such that $\|\Gamma^2(s^2) - \Gamma^2(s^1)\|_{\infty} \leq K \|s^2 - s^1\|_{\infty}$ whenever $R_1\unit \leq s^1 \leq s^2 \leq R_2\unit$.%
\item[(ii)] The cardinality of $I_i$ is bounded over $i \in \N$ and for every compact interval $J \subset (0,\infty)$, there exists $C > 0$ such that%
\begin{equation*}
  |\gamma_{ij}(r_1) - \gamma_{ij}(r_2)| \leq C |r_1 - r_2|,%
\end{equation*}
whenever $r_1,r_2 \in J$ and $i \in \N$, $j \in I_i$.%
\end{enumerate}
Then $\sigma(r) := \sigma_*(r) = \sigma^*(r)$ is a path of decay for $\Gamma$.
\end{theorem}

\begin{proof}
We first prove that for any $s,s^1,s^2 \in \ell_{\infty}^+$ satisfying $s^1 \leq s^2$, we have%
\begin{equation}\label{eq_sumtype_ineq}
  \|\Gamma(s \oplus s^2) - \Gamma(s \oplus s^1)\|_{\infty} \leq \|\Gamma(s^2) - \Gamma(s^1)\|_{\infty}.%
\end{equation}
For each $i \in \N$, we have%
\begin{align*}
  \Gamma_i(s \oplus s^2) - \Gamma_i(s \oplus s^1) &= \sum_{j\in I_i}\gamma_{ij}(\max\{s_j,s_j^2\}) - \sum_{j\in I_i}\gamma_{ij}(\max\{s_j,s_j^1\}) \\
	&= \sum_{j \in I_i} \Bigl( \max\{ \gamma_{ij}(s_j),\gamma_{ij}(s_j^2) \} - \max\{ \gamma_{ij}(s_j),\gamma_{ij}(s^1_j) \} \Bigr) \\
	&\leq \sum_{j \in I_i} \max\{ \gamma_{ij}(s_j) - \gamma_{ij}(s_j), \gamma_{ij}(s_j^2) - \gamma_{ij}(s_j^1) \} \\
	&= \sum_{j \in I_i} \max\{ 0, \gamma_{ij}(s_j^2) - \gamma_{ij}(s_j^1) \} \\
	&= \sum_{j \in I_i} (\gamma_{ij}(s_j^2) - \gamma_{ij}(s_j^1)) = \Gamma_i(s^2) - \Gamma_i(s^1).%
\end{align*}
This implies the desired inequality. Now, fix a compact interval $J \subset (0,\infty)$ and choose $0 < R_1 < R_2$ such that $R_1 \leq r \leq \varphi^{-1}(R_2)$ for all $r \in J$. Let $K$ be the contraction constant of $\Gamma^2$ on $[R_1\unit,R_2\unit]$. For any $0 < r_1 < r_2$ in $J$ and $\sigma := \sigma_*$, it follows that%
\begin{align*}
  &\|\sigma(r_2) - \sigma(r_1)\|_{\infty} = \|\Gamma_{r_2}^2(\sigma(r_2)) - \Gamma_{r_1}^2(\sigma(r_1))\|_{\infty} \\
	&\leq \|\Gamma_{r_2}^2(\sigma(r_2)) - \Gamma_{r_2}^2(\sigma(r_1))\|_{\infty} + \|\Gamma_{r_2}^2(\sigma(r_1)) - \Gamma_{r_1}^2(\sigma(r_1))\|_{\infty} \\
	&\leq K\|\sigma(r_2) - \sigma(r_1)\|_{\infty} + \max\{1,L_{\Gamma}\} |r_1 - r_2|,%
\end{align*}
where $L_{\Gamma}$ is a Lipschitz constant of $\Gamma$ on $[R_1\unit,R_2\unit]$ whose existence easily follows from Assumption (ii). Here, exploiting the contraction property of the projections $P_r$ (see Proposition \ref{prop_pr_props}), we use that%
\begin{align*}
  \|\Gamma_{r_2}^2(\sigma(r_2)) - \Gamma_{r_2}^2(\sigma(r_1))\|_{\infty} &\leq \|\Gamma(r_2\unit \oplus \Gamma(\sigma(r_2))) - \Gamma(r_2\unit \oplus \Gamma(\sigma(r_1)))\|_{\infty} \\
	&\stackrel{\eqref{eq_sumtype_ineq}}{\leq} \|\Gamma^2(\sigma(r_2)) - \Gamma^2(\sigma(r_1))\|_{\infty}%
\end{align*}
and%
\begin{align*}
  \Gamma_{r_2}^2(\sigma(r_1)) - \Gamma_{r_1}^2(\sigma(r_1)) &\leq (r_2 - r_1)\unit \oplus (\Gamma(r_2\unit \oplus \Gamma(\sigma(r_1))) - \Gamma(r_1\unit \oplus \Gamma(\sigma(r_1)))) \\
	&\stackrel{\eqref{eq_sumtype_ineq}}{\leq} (r_2 - r_1)\unit \oplus (\Gamma(r_2\unit) - \Gamma(r_1\unit)),%
\end{align*}
which implies%
\begin{align*}
  \|\Gamma_{r_2}^2(\sigma(r_1)) - \Gamma_{r_1}^2(\sigma(r_1))\|_{\infty} \leq \max\{ |r_2 - r_1|, \|\Gamma(r_2\unit) - \Gamma(r_1\unit)\|_{\infty}\}.%
\end{align*}
It then follows that%
\begin{equation*}
  \|\sigma(r_2) - \sigma(r_1)\|_{\infty} \leq \frac{\max\{1,L_{\Gamma}\}}{1 - K}|r_1 - r_2| \mbox{\quad for all\ } r_1,r_2 \in J,%
\end{equation*}
showing that $\sigma$ is locally Lipschitz continuous. It is easy to see that our assumption on $\Gamma^2$ also implies that $\Gamma_r$ admits only one fixed point, and so $\sigma_* = \sigma^*$. The proof is complete.
\end{proof}

\begin{remark}
For general gain operators, it can be proved similarly that a path of decay exists if the assumptions (i)--(v) of Theorem \ref{thm_path_of_decay} are satisfied and $\Gamma_r^k$ is an order contraction for some $k \geq 1$. However, this condition is hard to check, in general.
\end{remark}

\begin{remark}
If we replace the gain operator $\Gamma$ by a slightly enlarged operator $\Gamma_{\omega} = \omega^{-1} \circ \Gamma$, where $\omega < \id$ is a $\KC_{\infty}$-function, and impose the assumptions of Theorem \ref{thm_path_of_decay} on this operator, it probably suffices to find a continuous selection of the set-valued mapping $r \mapsto \Fix(\Gamma_{\omega,r})$. The reason is that we can try to approximate such a selection $\sigma$ by a locally Lipschitz continuous path which satisfies properties (P2), (P3) and (P4) of a path of decay. Property (P1) can also be satisfied because of the little extra space gained by $\omega$. Indeed, $\sigma(r) = r\unit \oplus \omega^{-1}(\Gamma(\sigma(r)))$, and thus $\Gamma(\sigma(r)) \leq \omega(\sigma(r)) < \sigma(r)$. We leave the technical details of this construction to a future work.
\end{remark}

\section*{Conclusions}

In this paper, we studied a very general class of gain operators associated with infinite networks of finite-dimensional control systems. We developed a new approach to find points of decay and paths of decay for such operators, which are used for the construction of ISS Lyapunov functions via a small-gain approach as described in Theorem \ref{thm_smallgain_result}.%

We have introduced an augmented gain operator whose fixed points are precisely the points of decay for the original gain operator. In Proposition \ref{prop_smallgain_conds}, we have shown that the augmented gain operator has plenty of fixed points under certain uniform small-gain conditions. We also proved in Proposition \ref{prop_uges_implies_mbi} that these uniform small-gain conditions hold if the system induced by a slightly enlarged gain operator is exponentially stable. In general, this is certainly a too strong assumption, as for max-type operators much weaker assumptions suffice, see Proposition \ref{prop_hatgamma_ugs}.%

As an important prerequisite for our results on the existence of a path of decay, we described the uniform global asymptotic stability of the system induced by a gain operator in terms of small-gain-like conditions, see Proposition \ref{prop_ugas_char_part1} and \ref{prop_ugas_char_part2}. Since we did not obtain a complete characterization of the UGAS property, some open questions remain which we leave for future work.%

Our main result about the existence of a path of decay, i.e.~Theorem \ref{thm_path_of_decay}, requires, in particular, that the system induced by a slightly enlarged gain operator is uniformly globally asymptotically stable. However, in contrast to the known results for finite networks, it also requires the existence of a Lipschitz continuous selection of a certain set-valued mapping. The question when such a selection exists, in the general case, has not been answered. However, we showed that in several cases studied previously in the literature, our approach is able to recover the known results. Further studies are needed to make the approach fruitful for more general types of gain operators.

\bibliographystyle{abbrv}
\bibliography{literature}

\begin{thebibliography}{10}

\bibitem{bamieh}
B.~Bamieh, M.~R. Jovanovic, P.~Mitra, and S.~Patterson.
\newblock Coherence in large-scale networks: Dimension-dependent limitations of
  local feedback.
\newblock {\em IEEE Transactions on Automatic Control}, 57(9):2235--2249, 2012.

\bibitem{dashkovskiy2019stability}
S.~Dashkovskiy, A.~Mironchenko, J.~Schmid, and F.~Wirth.
\newblock Stability of infinitely many interconnected systems.
\newblock {\em IFAC-PapersOnLine}, 52(16):550--555, 2019.

\bibitem{dashkovskiy2020stability}
S.~Dashkovskiy and S.~Pavlichkov.
\newblock Stability conditions for infinite networks of nonlinear systems and
  their application for stabilization.
\newblock {\em Automatica}, 112:108643, 2020.

\bibitem{dashkovskiy2007iss}
S.~Dashkovskiy, B.~S. R{\"u}ffer, and F.~R. Wirth.
\newblock An {ISS} small gain theorem for general networks.
\newblock {\em Mathematics of Control, Signals, and Systems}, 19(2):93--122,
  2007.

\bibitem{dashkovskiy2010small}
S.~N. Dashkovskiy, B.~S. R{\"u}ffer, and F.~R. Wirth.
\newblock Small gain theorems for large scale systems and construction of {ISS}
  {L}yapunov functions.
\newblock {\em SIAM Journal on Control and Optimization}, 48(6):4089--4118,
  2010.

\bibitem{gluck2021stability}
J.~Gl{\"u}ck and A.~Mironchenko.
\newblock Stability criteria for positive linear discrete-time systems.
\newblock {\em Positivity}, 25(5):2029--2059, 2021.

\bibitem{KaJ11}
I.~Karafyllis and Z.-P. Jiang.
\newblock A vector small-gain theorem for general non-linear control systems.
\newblock {\em IMA Journal of Mathematical Control and Information},
  28:309--344, 2011.

\bibitem{kawan}
C.~Kawan, A.~Mironchenko, A.~Swikir, N.~Noroozi, and M.~Zamani.
\newblock A {L}yapunov-based small-gain theorem for infinite networks.
\newblock {\em IEEE Transactions on Automatic Control}, 66(12):5830--5844,
  2021.

\bibitem{kawan2021lyapunov}
C.~Kawan, A.~Mironchenko, and M.~Zamani.
\newblock A {L}yapunov-based {ISS} small-gain theorem for infinite networks of
  nonlinear systems.
\newblock {\em arXiv:2103.07439}, 2021.

\bibitem{mironchenko2017uniform}
A.~Mironchenko.
\newblock Uniform weak attractivity and criteria for practical global
  asymptotic stability.
\newblock {\em Systems \& Control Letters}, 105:92--99, 2017.

\bibitem{mironchenko2021nonlinear}
A.~Mironchenko, C.~Kawan, and J.~Gl{\"u}ck.
\newblock Nonlinear small-gain theorems for input-to-state stability of
  infinite interconnections.
\newblock {\em Mathematics of Control, Signals, and Systems}, 33(4):573--615,
  2021.

\bibitem{mironchenko2021iss}
A.~Mironchenko, N.~Noroozi, C.~Kawan, and M.~Zamani.
\newblock {ISS} small-gain criteria for infinite networks with linear gain
  functions.
\newblock {\em Systems \& Control Letters}, 157, November 2021.

\bibitem{MiP20}
A.~Mironchenko and C.~Prieur.
\newblock Input-to-state stability of infinite-dimensional systems: recent
  results and open questions.
\newblock {\em SIAM Review}, 62(3):529--614, 2020.

\bibitem{ruffer2010monotone}
B.~S. R{\"u}ffer.
\newblock Monotone inequalities, dynamical systems, and paths in the positive
  orthant of {E}uclidean n-space.
\newblock {\em Positivity}, 14(2):257--283, 2010.

\bibitem{Sch20}
F.~L. Schwenninger.
\newblock Input-to-state stability for parabolic boundary control: linear and
  semilinear systems.
\newblock In {\em Control Theory of Infinite-Dimensional Systems}, pages
  83--116. Springer, 2020.

\end{thebibliography}

\end{document}